%% file: Tripods.tex
\documentclass[11pt]{amsart}

\usepackage[margin=1in,letterpaper,portrait]{geometry}
\usepackage{amsmath}
\usepackage{amssymb}
\usepackage{amsthm}
\usepackage{verbatim}
\usepackage{fancyhdr}
\usepackage{url}
\usepackage{tikz}
\usepackage{tkz-euclide}
\usepackage{wrapfig}
\usepackage{enumerate}
\usetikzlibrary{calc,intersections}
\usepackage{float}
\usepackage{subfigure}
\usepackage{bookmark}
\usepackage{hyperref}
\usepackage[normalem]{ulem}

\newcommand{\mb}{\mathbb}

\newcommand{\mc}{\mathcal}

\begin{document}
\theoremstyle{plain} 
\newtheorem{Thm}{Theorem}[section]
\newtheorem*{Thm*}{Theorem} 
\newtheorem{Cor}[Thm]{Corollary}
\newtheorem*{Cor*}{Corollary}
\newtheorem{Main}{Main Theorem}
\renewcommand{\theMain}{} 
\newcommand{\Sol}{\text{Sol}}
\newtheorem{Lem}[Thm]{Lemma}
\newtheorem{Prop}[Thm]{Proposition} 
\newtheorem{Conj}[Thm]{Conjecture}

\theoremstyle{definition} 
\newtheorem{Def}{Definition}
\newtheorem{Exer}{Exercise} 
\newtheorem{Prob}{Problem}
\newtheorem*{Rem}{Remark}

\theoremstyle{remark} 


\newcommand{\pd}[2]{\frac{\partial #1 }{\partial #2 }}
\newcommand{\grad}{\bigtriangledown}
\newcommand{\revision}[1]{\textbf{\color[rgb]{1, 0, 0} #1}}

\title{Tripod Configurations of Curves}

\author{Eric Chen}
\address{Princeton University}
\email{ecchen@princeton.edu}

\author{Nick Lourie}
\address{Brown University}
\email{nalourie@gmail.com}

\begin{abstract}
Tripod configurations of plane curves, formed by certain triples of normal lines coinciding at a point, were introduced by Tabachnikov, who showed that $C^2$ closed convex curves possess at least two tripod configurations. Later, Kao and Wang established the existence of tripod configurations for $C^2$ closed locally convex curves. In this paper we generalize these two results, answering a conjecture of Tabachnikov on the existence of tripod configurations for all $C^2$ closed plane curves by proving existence for a generalized notion of tripod configuration. We then demonstrate the existence of the natural extensions of these tripod configurations to the spherical and hyperbolic geometries for a certain class of convex curves, and discuss an analogue of the problem for regular plane polygons. 
\end{abstract}

\maketitle

\section{Introduction}\label{IN}

As introduced by Tabachnikov in \cite{ST}, a tripod configuration for a $C^2$ convex closed curve $\gamma$ consists of three perpendicular lines dropped from three points on $\gamma$ meeting at a single point and together making angles of $2\pi/3$. Tabachnikov proved that all $C^2$ closed convex curves have at least two tripod configurations, and later Kao and Wang \cite{KW} showed the existence and provided a lower bound on the number of tripod configurations as defined for locally convex curves.

Below in this section we fix the definitions for tripod configurations of closed curves in the plane to be used later. In Section \ref{PR} we present previous results on tripod configurations; curves are assumed to be $C^2$ and closed unless otherwise specified. In Section \ref{loca} we give an improved lower bound over the result in \cite{KW} for the number of tripod configurations possessed by locally convex plane curves. In Section \ref{TTNIT} we generalize the notion of tripod configurations to ``triple normals,'' and show that these ``triple normals'' exist for any $C^2$ closed plane curve. It follows as a corollary of this result that every $C^2$ closed plane curve (including immersed curves possibly with self intersections) has at least one tripod configuration. In Sections \ref{MTIntro} through \ref{CMT} we extend the notion of tripod configurations to the spherical and hyperbolic geometries and show the existence of tripod configurations for convex closed curves sufficiently close to circles.  Finally, in Section \ref{poly} we describe an extension of the notion of tripod configurations for regular polygons.

The consideration of tripod configurations arises from the discussion of the four vertex theorem and related results from Tabachnikov \cite{ST}. Tripod configurations are also natural generalizations of two classical notions, the Fermat-Toricelli point and double normals of closed curves. 

The Fermat-Toricelli point of a triangle is the unique point minimizing the sum of the distances from the three vertices of the triangle to the point; when no angle of the triangle exceeds $2\pi/3$, the lines from the triangle vertices to the Fermat-Toricelli point form angles of $2\pi/3$; otherwise, the Fermat-Toricelli point coincides with a triangle vertex. Thus, the intersection of the three lines of a tripod configuration of convex closed curve $\gamma$ is exactly the Fermat-Toricelli point of the triangle with vertices given by the three points on $\gamma$ from which the perpendiculars are dropped.

\begin{figure}[H]  
  \centering  
  \def\svgwidth{0.25\columnwidth}  
  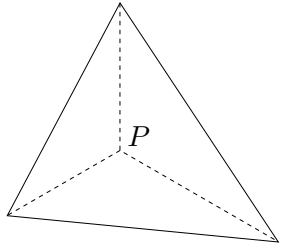
  \caption{The Fermat-Toricelli point $P$ of a triangle}  
\end{figure}

The study of double normals or diameters, chords of closed curves (or surfaces) meeting the curves orthogonally, appears for example in \cite{BH} and \cite{NK}, in which lower bounds and formulae connecting double normals and tangency lines are established. Double normals also arise in the context of curves of constant width; for instance, see \cite{YB}. Tripod configurations are then an extension to three normals with ``nice'' meeting (evenly spaced angles of $2\pi/3$) from the double normals setting with two coincident normals.

We next consider the following definitions for tripod configurations of $C^2$ closed plane curves.

\begin{Def}\label{angle}
Given a closed plane curve $\gamma$, a \emph{tripod configuration} of $\gamma$ consists of three lines normal to the curve all coincident at a single point and together forming three angles of $2\pi/3$.
\end{Def}

\begin{Def}\label{vector}
Given a closed plane curve $\gamma$, a \emph{tripod configuration} of $\gamma$ consists of three lines normal to the curve all coincident at a single point such that the sum of the three unit normal vectors (oriented according to the curve orientation) associated to the three normal lines is $0$.
\end{Def}

\begin{figure}[H]
\centering
\subfigure[Definitions \ref{angle} and \ref{vector}]{\includegraphics[scale=0.2]{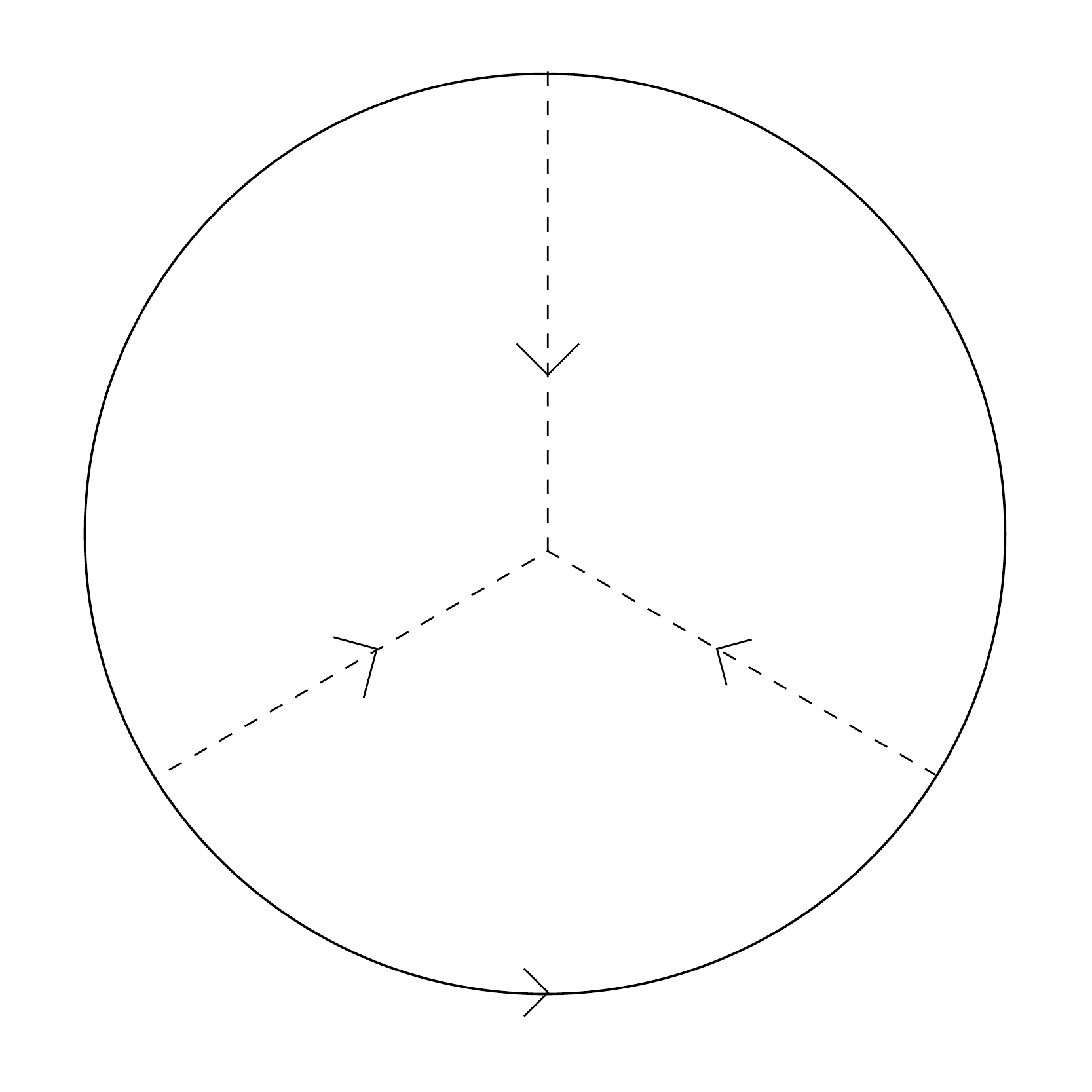}\label{def1}}
\quad
\subfigure[Definition \ref{angle}]{\includegraphics[scale=0.2]{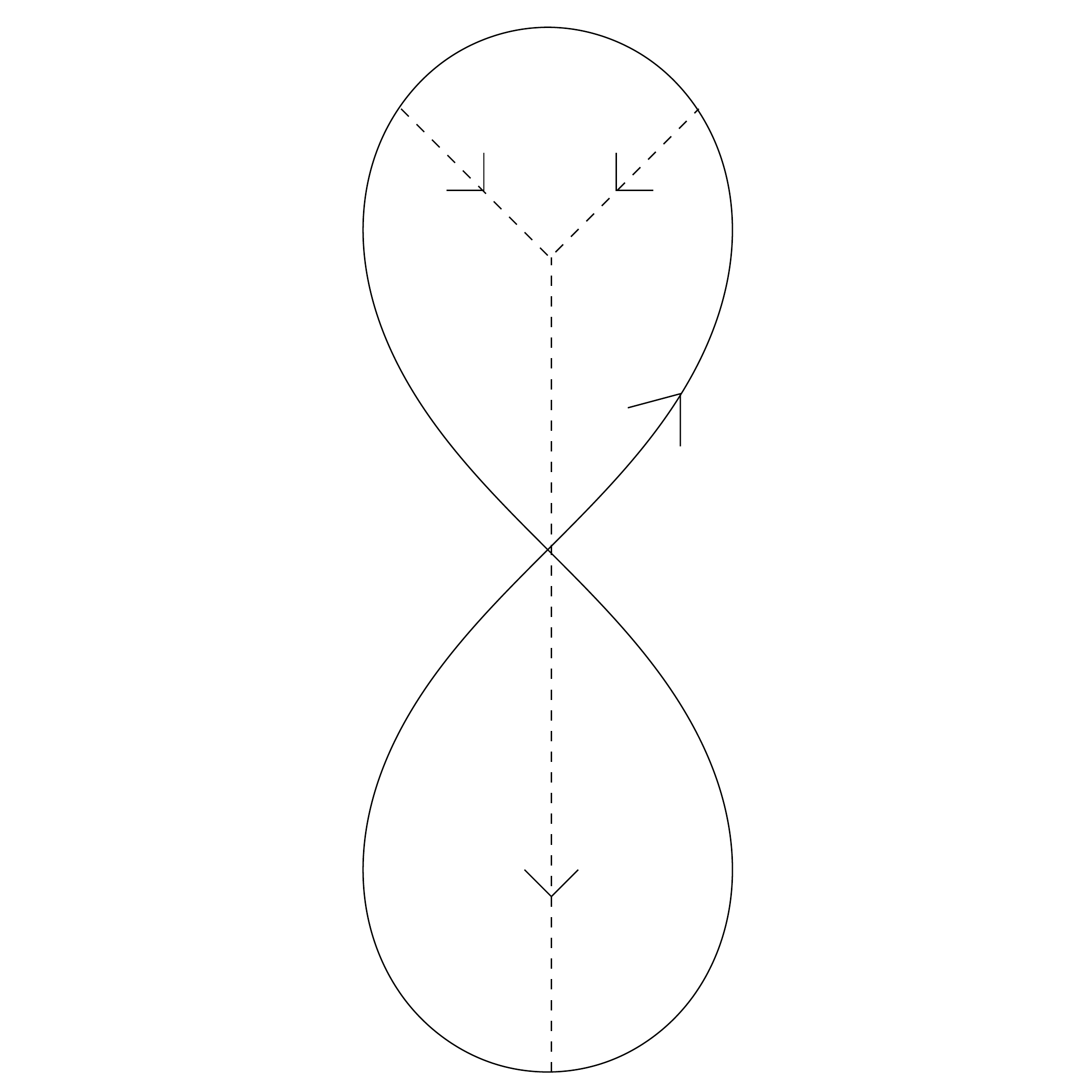}\label{def2}}
\quad
\subfigure[Definitions \ref{angle} and \ref{vector}]{\includegraphics[scale=0.2]{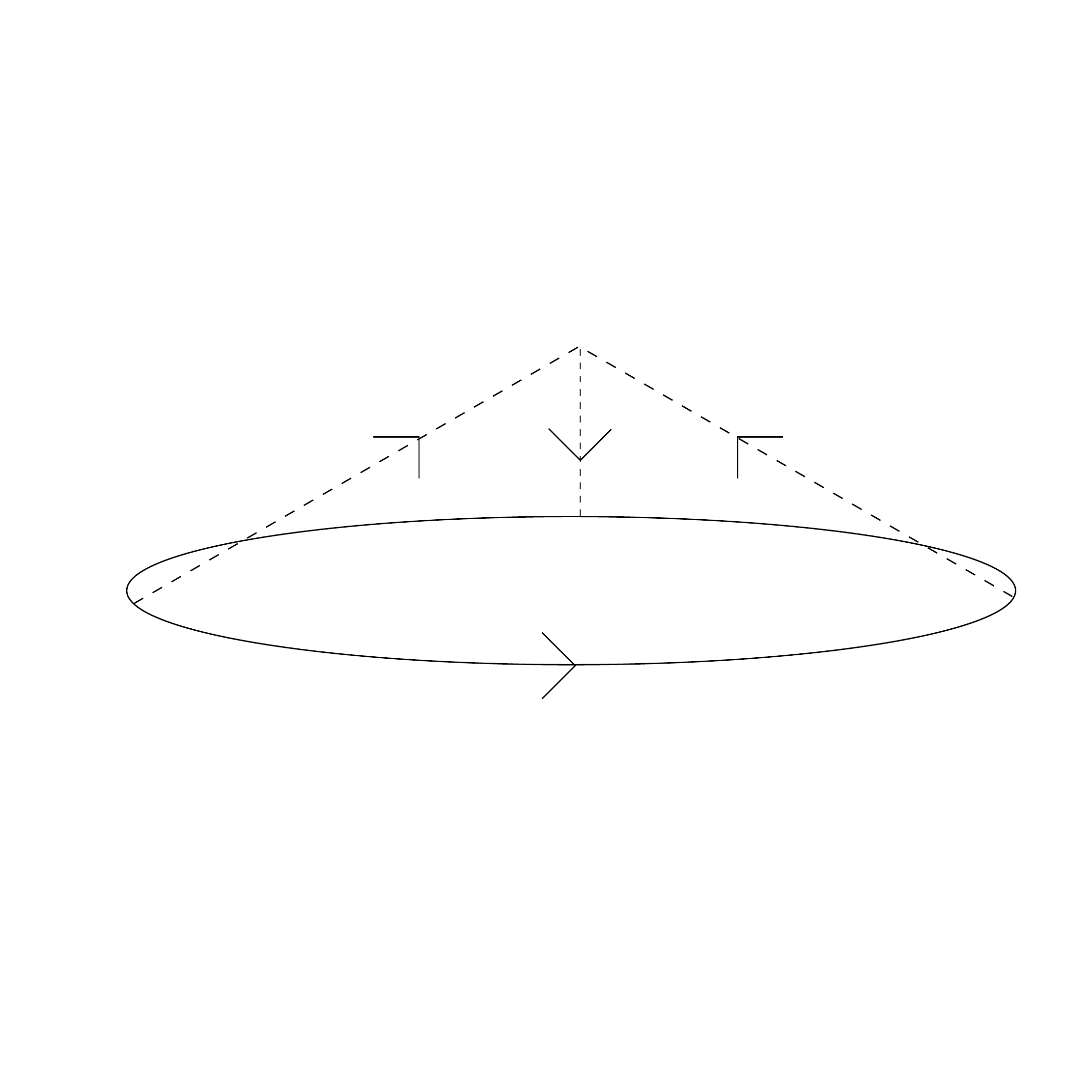}\label{def3}}
\caption{Tripod configurations}\label{tripod}
\end{figure}

Definition \ref{angle} is more general than Definition \ref{vector}; in Figure \ref{tripod}, each subfigure is labeled with the definitions it satisfies. In both \cite{ST} and \cite{KW} Definition \ref{angle} is the one explicitly stated and motivated as discussed above, while the stronger Definition \ref{vector} is the one used in the proofs of the theorems. In proving our results below we will use the original Definition \ref{angle}.

Finally, for convenience, we fix the following definition.

\begin{Def}\label{tripodpoint}
A \textit{tripod point} is the point at which three lines in a tripod configuration intersect.
\end{Def}

In particular, a single tripod point may be associated with many, even infinitely many, tripod configurations, as in the case of a circle. 

\section{Prior results and statement of theorems}\label{PR}

In this section we present the main theorems of \cite{ST} and \cite{KW}, and then list the five results we prove in this paper. We omit the proof of the lower bound estimate in Theorem \ref{local} from \cite{KW}, since in Theorem \ref{bound} we will give a precise count for what is estimated there.

\begin{Thm}[Tabachnikov \cite{ST}]\label{convex}
Every smooth ($C^2$) closed convex plane curve has at least two tripod configurations.
\end{Thm}

\begin{proof}
Let $\gamma(s)$ be an arc length parametrization of the curve, and let $t(s)=\gamma'(s)$, $n(s)=\gamma''(s)/|\gamma''(s)|$,  the tangent and normal unit vectors to the curve at $\gamma(s)$, respectively. Further let $\alpha(s)$ be the angle made by $\gamma'(s)$ with some fixed direction. We may also parametrize the curve by $\alpha$, so that $\gamma(\alpha),t(\alpha),n(\alpha)$ are in analogy to the above. Define

\begin{align*}
q(\alpha)&=\gamma(\alpha)\cdot n(\alpha),
\\
p(\alpha)&=\frac{d}{d\alpha}q(\alpha)=-\gamma(\alpha)\cdot t(\alpha).
\end{align*}

Let $\ell(\alpha)$ and $\bar{\ell}(\alpha)$ denote the normal and tangent lines to the curve at $\gamma(\alpha)$, respectively. Then the equilateral triangle bounded by the normal lines $\ell(\alpha),\ell(\alpha+2\pi/3),\ell(\alpha+4\pi/3)$ has area

\begin{align*}
\frac{1}{\sqrt{3}}(p(\alpha)+p(\alpha+2\pi/3)+p(\alpha+4\pi/3))^2,
\end{align*}

and the equilateral triangle bounded by the tangent lines $\bar{\ell}(\alpha),\bar{\ell}(\alpha+2\pi/3),\bar{\ell}(\alpha+4\pi/3)$ has area

\begin{align*}
\frac{1}{\sqrt{3}}(q(\alpha)+q(\alpha+2\pi/3)+q(\alpha+4\pi/3))^2.
\end{align*}

Thus a tripod configuration is achieved by $\ell(\alpha),\ell(\alpha+2\pi/3),\ell(\alpha+4\pi/3)$ exactly when $p(\alpha)+p(\alpha+2\pi/3)+p(\alpha+4\pi/3)=0$, which happens at least twice since $q(\alpha)+q(\alpha+2\pi/3)+q(\alpha+4\pi/3)$ is periodic and attains a maximum and minimum.

\end{proof}

\begin{Rem}
In the proof above, $\bar{\ell}(\alpha),\bar{\ell}(\alpha+2\pi/3),\bar{\ell}(\alpha+4\pi/3)$ form an equilateral triangle that circumscribes the curve $\gamma$, and $\alpha$ at which this triangle attains a local maximum or minimum area are such that $\ell(\alpha),\ell(\alpha+2\pi/3),\ell(\alpha+4\pi/3)$ form tripod configurations of the curve.

The functions $q(\alpha)$ and $p(\alpha)$ defined in the proof above both have useful geometric interpretations: $q(\alpha)$ is the ``support function,'' which for a convex closed curve $\gamma$ represents the distance from an origin chosen inside the region enclosed by the curve to the (unique) tangent line on $\gamma$ in the positive $\alpha$ direction, and $p(\alpha)$ is the distance from the origin to the normal line to $\gamma$ associated with the tangent line of $q(\alpha)$.

The support function $q(\alpha)$ may in fact be used to define the original closed convex curve $\gamma$. And using $q(\alpha)$ we may easily define curves equidistant to $\gamma$-that is, curves of the form $\gamma_r(s)=\gamma(s)+rn(s)$, by defining the corresponding support function $q_r(\alpha)=q(\alpha)+r$. From this property it follows that $\gamma$ has the same tripod configurations as its convex equidistant curves; in fact it is easy to see that this holds regardless of whether the equidistant curves are convex or not, allowing for wavefronts with co-orientation, despite cuspidal singularities.
\end{Rem}

\begin{figure}[H]  
  \centering  
  \def\svgwidth{0.3\columnwidth}  
  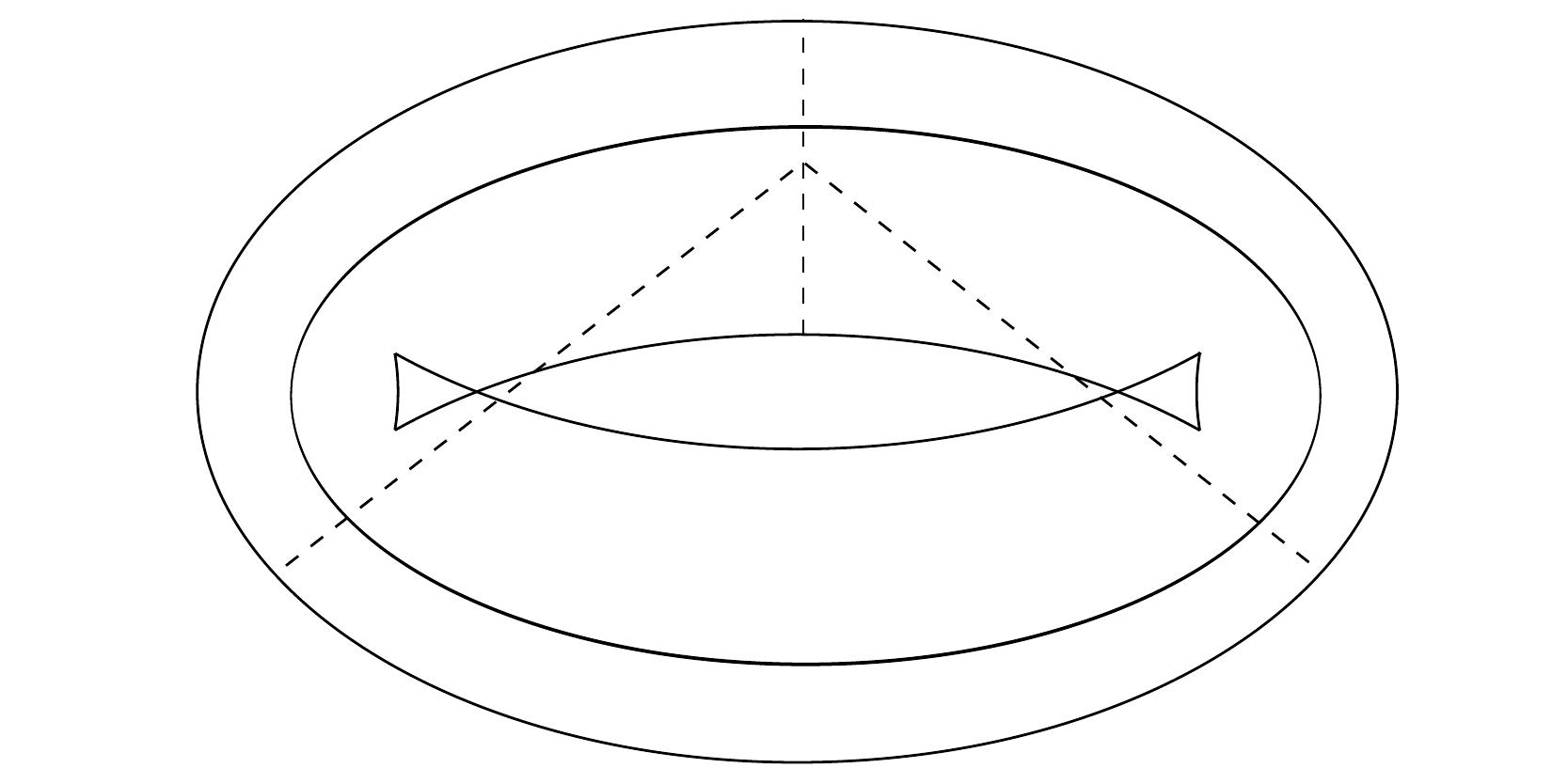
  \caption{Equidistant curves share normal lines and tripod configurations}  
\end{figure}

\begin{Rem}
When $q(\alpha)+q(\alpha+2\pi/3)+q(\alpha+4\pi/3)$ is constant, or equivalently when the equilateral triangles circumscribing $\gamma$ defined by the tangent lines at $\gamma(\alpha),\gamma(\alpha+2\pi/3),\gamma(\alpha+4\pi/3)$ have constant area, then by the proof of Theorem \ref{convex} every line normal to $\gamma$ belongs to a tripod configuration of $\gamma$. Properties of these closed curves, called $\Delta$-curves, which may be rotated freely inside an equilateral triangle in contact with all its sides, may be found in \cite{YB}. 
\end{Rem}

\begin{Thm}[Kao and Wang \cite{KW}]\label{local}
A smooth ($C^2$) closed locally convex plane curve with rotation index $n$ has at least $n^2/3$ tripod configurations.
\end{Thm}

\begin{proof}
Let $\gamma$ be a smooth closed locally convex plane curve with rotation index $n$. We may parametrize $\gamma$ by the angle $\alpha$ which $\gamma'$ makes with a fixed direction for $\alpha\pmod{2\pi n}$. Let $q$, $p$ be as defined in the proof of Theorem \ref{convex}. Define

\begin{align*}
P_{i,j,k}(\alpha)=p(\alpha+2\pi i/3)+p(\alpha+2\pi j/3)+p(\alpha+2\pi k/3),
\end{align*}

where $i,j,k\in\{0,1,\ldots,3n\}$ and $\{i,j,k\}\equiv\{0,1,2\}\pmod{3}$. By the same argument as in the proof of Theorem \ref{convex}, the critical points of the functions $P_{i,j,k}$ correspond to tripod configurations of $\gamma$, and the minimum number of tripod configurations is then twice the number of $P_{i,j,k}$ distinct. The problem becomes to count the number of distinct classes of such $P_{i,j,k}$ with $i,j,k\in\{0,1,\ldots,3n\}$ and $\{i,j,k\}\equiv\{0,1,2\}\pmod{3}$ given by the relation $P_{i,j,k}\sim P_{i',j',k'}$ whenever $\{i,j,k\}\equiv\{i'+m,j'+m,k'+m\}$ for some $m\in\mb{Z}$. See \cite{KW} for the proof of the lower bound $n^2/6$ on the number of such distinct classes.
\end{proof}

We conclude by stating the results we will prove in the remainder of this paper. The last two theorems require some further definitions which will be discussed in their respective sections.

\begin{Thm*}[Theorem \ref{bound}]
A smooth ($C^2$) closed locally convex plane curve with rotation index $n$ has at least $2\lceil \frac{n^2+2}{3} \rceil$ tripod configurations, where $\lceil \cdot \rceil$ denotes the greatest integer or ceiling function.
\end{Thm*}

\begin{Thm*}[Theorem \ref{TNIT}]
Given a smooth ($C^2$) closed plane curve $\gamma$ and three angles $\theta_1, \theta_2, \theta_3$, such that $\theta_1 + \theta_2 + \theta_3 = 2\pi$ and $\theta_1, \theta_2, \theta_3 < \pi$, there exist three normals to $\gamma$ intersecting at a single point and forming angles $\theta_1, \theta_2,$ and $\theta_3$. 
\end{Thm*}

\begin{Cor*}[Corollary \ref{pln}]
A smooth ($C^2$) closed plane curve has at least one tripod configuration. In particular, immersed plane curves with self intersection also possess at least one tripod configuration.
\end{Cor*}

\begin{Thm*}[Theorem \ref{sphyp}]
Every smooth ($C^2$) closed curve sufficiently close to a circle (excluding great circles on the sphere) in the spherical or hyperbolic geometry has at least two tripod configurations.
\end{Thm*}

\begin{Thm*}[Theorem \ref{regpoly}]
A regular $n$-vertex polygon has $n$ tripod configurations if $3\nmid n$ and has $n/3$ tripod configurations if $3\mid n$.
\end{Thm*}

\section{An improved lower bound for Theorem \ref{local}}\label{loca}
In this section we count the number of distinct functions $P_{i,j,k}$ described in the proof of Theorem \ref{local} above to establish an improved lower bound over that in Theorem \ref{loca}.

\begin{Thm}\label{bound}
A smooth ($C^2$) closed locally convex plane curve with rotation index $n$ has at least $2\lceil \frac{n^2+2}{3} \rceil$ tripod configurations, where $\lceil \cdot \rceil$ denotes the greatest integer or ceiling function.
\end{Thm}

\begin{proof}
Counting the number of distinct functions $P_{i,j,k}$ reduces to counting the number of distinct classes of $\{i,j,k\}\subset\{0,1,\ldots,3n-1\}$ with $\{i,j,k\}\equiv\{0,1,2\}\pmod{3}$ under the relation $\{i,j,k\}\equiv\{i',j',k'\}$ whenever $\{i,j,k\}=\{i'+m,j'+m,k'+m\}\pmod{3n}$ for some $m\in\mb{Z}$.

To count these equivalence classes, we may consider only those $\{i,j,k\}$ as specified above with $i=0$, $j\equiv 1\pmod{3}$, and $k\equiv 2\pmod{3}$, and count the distinct number of these modulo the given equivalence relation; indeed, given $\{i,j,k\}$ with $i,j,k$ being $0,1,2$ modulo $3$ respectively, we have $\{i,j,k\}\equiv\{i-i,j-i,k-i\}$, where $i-i=0$, $j-i\equiv 1\pmod{3}$, and $k-i\equiv 2\pmod{3}$. And each such $\{0,j,k\}$ is equivalent to at most two others of the same form; in general the following three are equivalent: $\{0,j,k\},\{0,k-j,-j\},\{0,-k,j-k\}$. Whether any of these might be pairwise equivalent occurs only for $\{0,n,2n\}$ if $3\nmid n$, and never if $3|n$ (because then all of $0,n,2n$ are divisible by $3$ so that $\{0,n,2n\}$ is not one of the sets under consideration). Accounting for the duplicates out of the $n^2$ possible sets $\{0,j,k\}$ satisfying the desired conditions gives the result.
\end{proof}

\begin{Rem}
The case $n=1$ is exactly the case of smooth closed convex plane curves addressed in Theorem \ref{convex}; the only function of the form $P_{i,j,k}$ is $P_{0,1,2}$, and there are at least two tripod configurations. The lower bound given in Theorem \ref{bound} is sharp in the case $n=1$ \cite{KW}, but it is unknown whether it is sharp in general. This improved lower bound also extends to the setting of the existence of tripod configurations for co-orientable wave fronts with total rotation $2\pi n$ by considering equidistant curves of locally convex curves, as mentioned earlier.
\end{Rem}

\section{The Triple Normal Intersection Theorem}\label{TTNIT}

We begin by giving a definition of a generalization of the first isogonic center from classical geometry. The first isogonic center of a triangle $\triangle ABC$ is constructed as follows: take a circle about each of $\overline{AB}, \overline{BC},$ and $\overline{CA}$ such that each line segment cuts a chord with angular measure $4\pi/3$ from its corresponding circle, as shown in Figure \ref{FirstIsogonicCenter}. Then these circles all intersect at a point $I_1$, which is called the first isogonic center of $\triangle ABC$.

\begin{figure}[H]  
  \centering  
  \def\svgwidth{0.3\columnwidth}  
  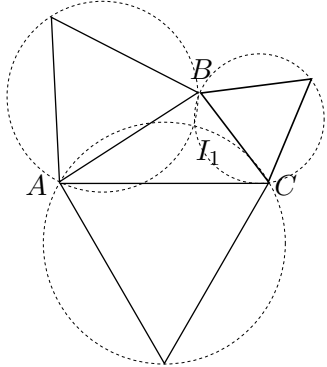
  \caption{The first isogonic center $I_1$ of $\triangle ABC$}
  \label{FirstIsogonicCenter}
\end{figure}

Let $\tau_1, \tau_2, \tau_3 \in [0, 2\pi)$ such that $\tau_1 + \tau_2 + \tau_3 = \pi$. Then we define the $\tau_1, \tau_2, \tau_3 -$centers of $\triangle ABC$ as follows:

\begin{Def}[$\tau_1, \tau_2, \tau_3 -$centers]\label{tcenterdef}
A $\tau_1, \tau_2, \tau_3 -$center of $\triangle ABC$ is constructed by forming circles around each of the chords $\overline{AB}$, $\overline{BC}$, and $\overline{CA}$ such that these chords cut arcs on the circles of measures $2\tau_1,2\tau_2$, and $2\tau_3$ lying on the same sides of the chords as points $C,A$, and $B$, respectively, as shown in Figure \ref{ABCwithArcs}. These circles will intersect at a point, and this point is a $\tau_1, \tau_2, \tau_3 -$center of $\triangle ABC$. Note that this point is only unique for fixed orderings of $A, B, C$.
\end{Def}

We'll now prove that the three circles described above intersect at a single point.

\begin{proof}
Construct three circles such that $\overline{AB}$, $\overline{BC}$, and $\overline{AC}$ form chords defining arcs of measure $2\tau_1$, $2\tau_2$, and $2\tau_3$, respectively.

Suppose the circles with chords $\overline{AB}$ and $\overline{BC}$ intersect at distinct points $P$ and $B$. Without loss of generality, we may assume that this intersection occurs in three cases. Either $P$ is closer to the chord $\overline{AC}$ than $B$, $P$ is the same point as $B$, or $P$ is farther away from $\overline{AC}$ than $B$.

\begin{figure}[H]  
  \centering  
  \subfigure[$P$ closer than $B$]{\def\svgwidth{0.3\columnwidth}  
  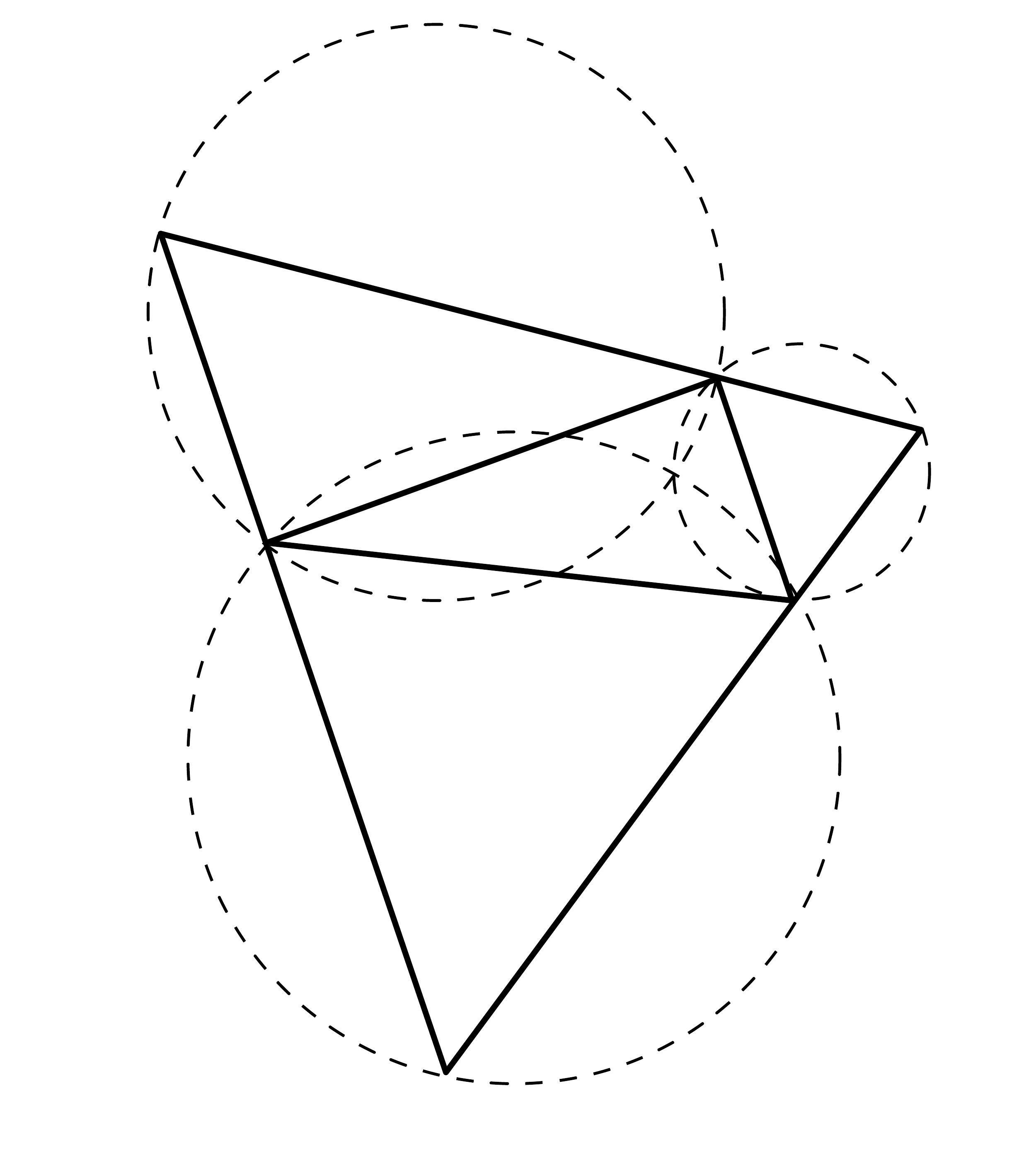}
  \quad
  \subfigure[$P = B$]{\def\svgwidth{0.3\columnwidth}  
  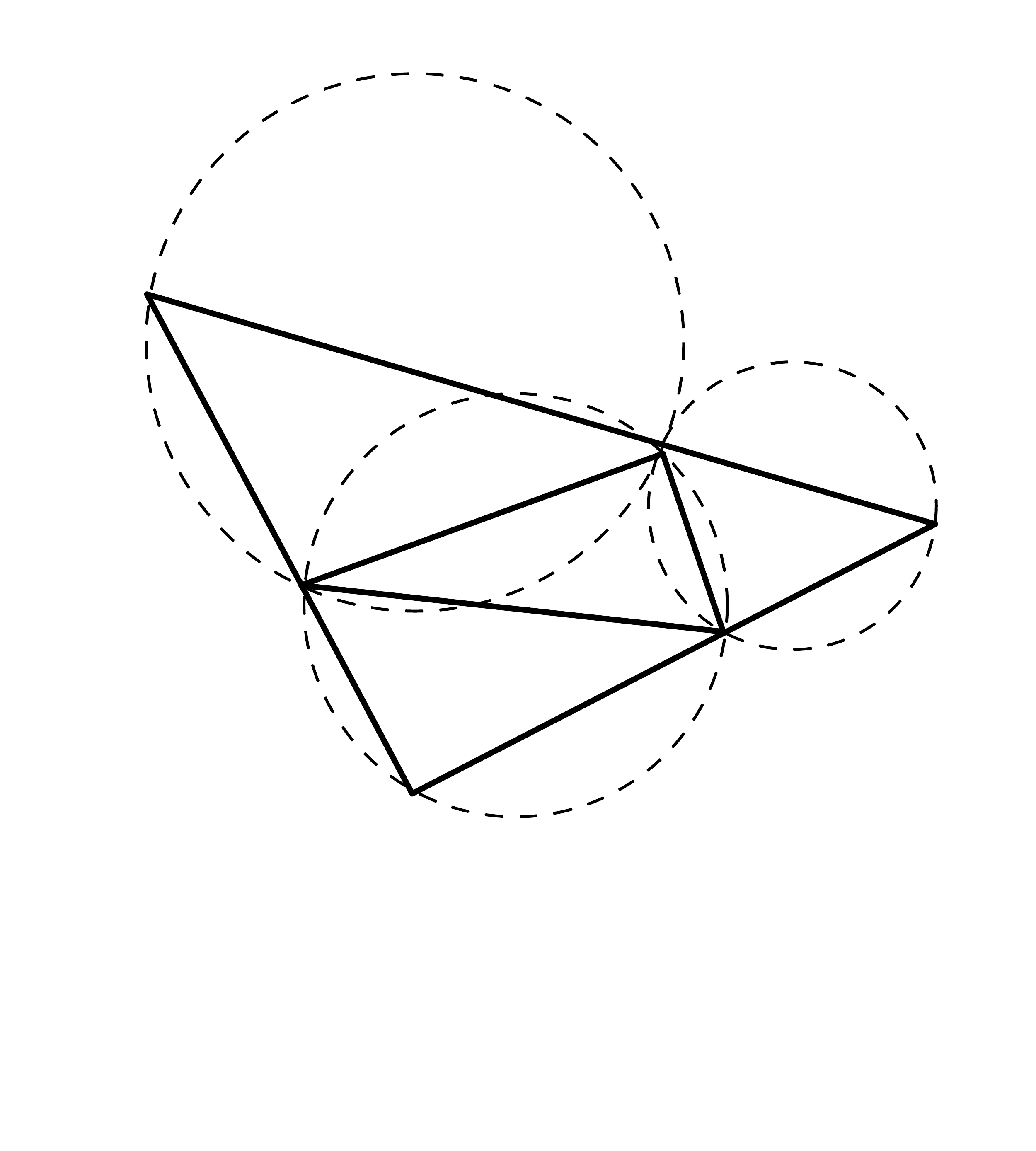}
  \quad
  \subfigure[$P$ farther than $B$]{\def\svgwidth{0.3\columnwidth}  
  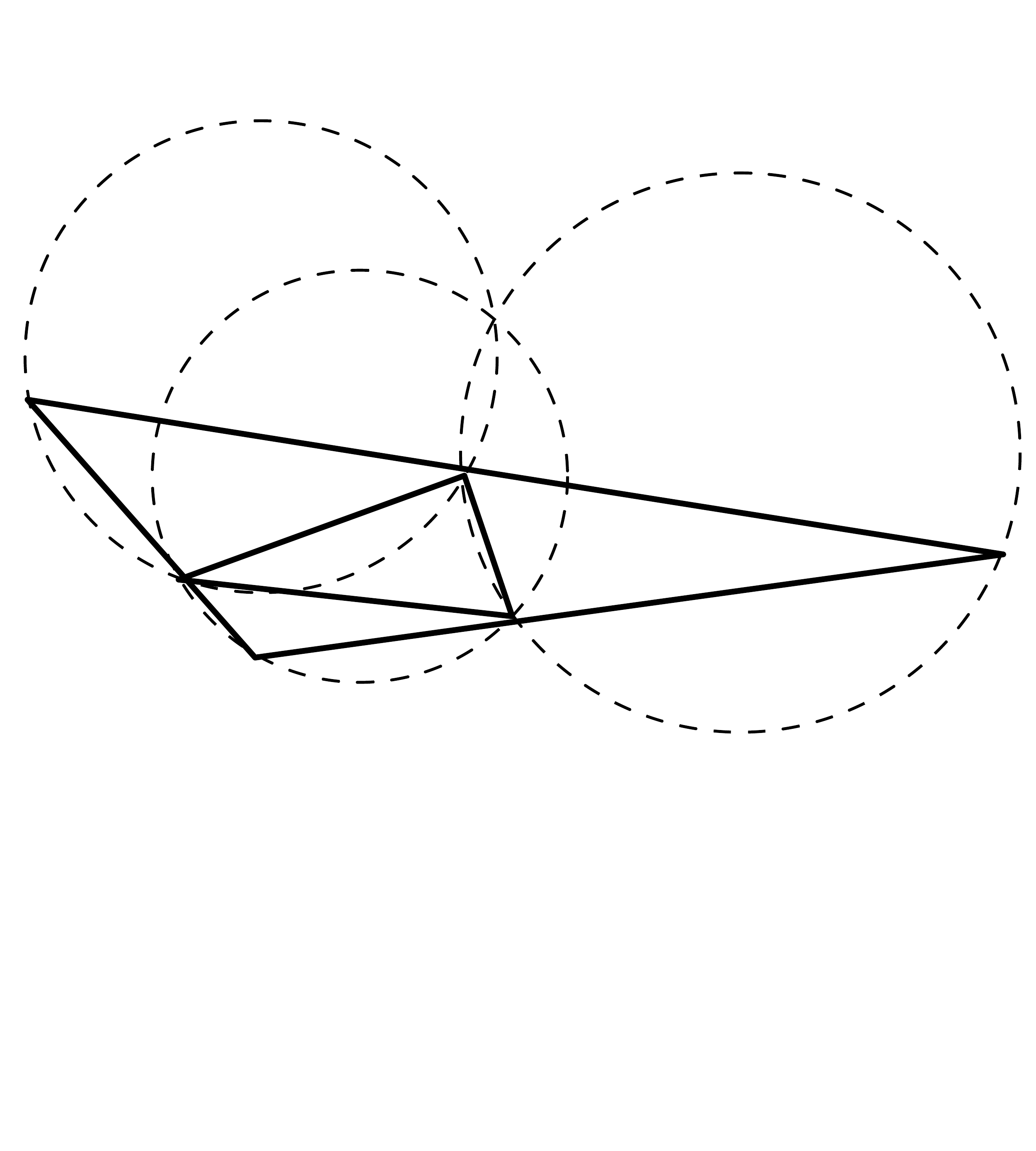}
  \caption{Three possible cases for position of $B$}
  \label{ABCwithArcs}
\end{figure}

By hypothesis, $P$ lies on the circles about $\overline{AB}$ and $\overline{BC}$, as one of their points of intersection. Notice that $\tau_3+\measuredangle CPA=\pi$, since: 

\begin{align*}
\measuredangle APB+\measuredangle BPC+\measuredangle CPA&=2\pi,
\\
(\pi-\tau_1)+(\pi-\tau_2)+\measuredangle CPA&=2\pi,
\\
(\pi-\tau_1-\tau_2)+\measuredangle CPA&=\pi.
\end{align*}

Therefore $P$ also lies on the circle about $\overline{AC}$.
\end{proof}

We need to recall a few more definitions before proving the central property of $\tau_1, \tau_2, \tau_3 -$centers.

\begin{Def}[Antipedal Triangle]
Given a triangle $ABC$ and a point $P$, \textit{the triangle antipedal to} $\triangle ABC$ \textit{with respect to} $P$ is the triangle with sides lying on the lines normal to $\overline{AP}, \overline{BP},$ and $\overline{CP}$ through the points $A, B,$ and $C$, respectively. 
\end{Def}

\begin{Def}
Define an equivalence relation on $\mathfrak{T}$, the set of planar triangles, by $\triangle ABC \sim \triangle DEF \iff$ $\triangle ABC$ and $\triangle DEF$ are similar with the same orientation (so that in general a triangle is not equivalent to its reflection). Then define $\mathcal{T} = \mathfrak{T} / \sim$, the set of equivalence classes of triangles under this relation.
\end{Def}

We now state and prove a key property of these generalized first isogonic centers:

\begin{Prop}\label{antipedalprop}
Given $T \in \mathcal{T}$ and $\triangle ABC$, any $\triangle DEF \in T$ of maximal area and circumscribing $\triangle ABC$ is the antipedal triangle to $\triangle ABC$ with respect to some $\tau_1, \tau_2, \tau_3 -$center, where $\tau_1, \tau_2, \tau_3$ are the angles of the vertices of $\triangle DEF$. 
\end{Prop}

\begin{proof}
Let $\tau_1, \tau_2, \tau_3$ be the angles of the vertices of the triangles in $T$, enumerated clockwise; we may assume $\triangle DEF\in T$ is oriented relative $\triangle ABC$ as in Figure \ref{DEFandABC}, relabeling vertices if necessary. 

\begin{figure}[H]  
  \centering  
  \def\svgwidth{0.3\columnwidth}  
  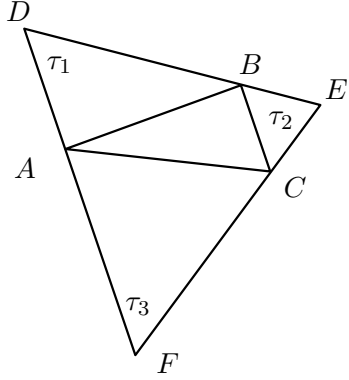
  \caption{$\triangle DEF$ and $\triangle ABC$}
  \label{DEFandABC}
\end{figure}

Perform the construction used to justify Definition \ref{tcenterdef} on $\triangle ABC$ with angles $\tau_1, \tau_2, \tau_3$, labeling as $\mathcal{O}_1$ the center of the circle circumscribing $\triangle DAB$ and $\mathcal{O}_2$ the center of the circle circumscribing $\triangle EBC$, as shown in Figure \ref{maximalSide}. For any point $D'$ sufficiently close to $D$ on the arc $ADB$, we may define $E'$ to be the intersection point of line $DB$ and arc $BEC$ distinct from $B$, and define $F'$ analogously with line $DA$ and arc $AFC$. Then $\triangle D'E'F'\in T$, and $\angle D'\equiv\angle D$, $\angle E'\equiv\angle E$, $\angle F'\equiv\angle F$. By hypothesis the length $D'E'$ should be maximized when $D'=D, E'=E$. 

\begin{figure}[H]  
  \centering  
  \def\svgwidth{0.3\columnwidth}  
  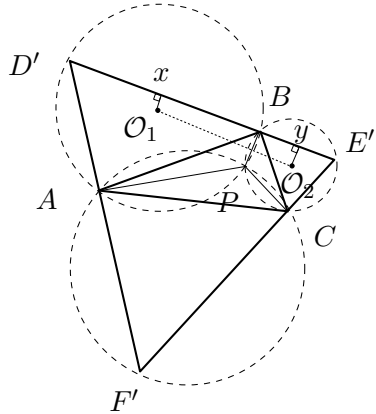
  \caption{Maximizing $\overline{D'E'}$}
  \label{maximalSide}
\end{figure}

Define $\overline{\mathcal{O}_1 x}$ and $\overline{\mathcal{O}_2 y}$ to be the shortest lines from $\mathcal{O}_1$ and $\mathcal{O}_2$ to the line segment $\overline{D'E'}$, respectively; $\overline{\mathcal{O}_1 x}$ perpendicularly bisects $\overline{D'B}$ and $\overline{\mathcal{O}_2 y}$ perpendicularly bisects $\overline{BE'}$. So the length of $\overline{D'E'}$ is twice the length of $\overline{xy}$, and $\overline{xy}$ is maximal when it is parallel to $\overline{\mathcal{O}_1 \mathcal{O}_2}$, thus perpendicular to $\overline{PB}$. Therefore $\overline{DE}$ is perpendicular to $\overline{PB}$. Repeating this argument on all three sides of the triangle shows that $\triangle DEF$ is antipedal to $\triangle ABC$ with respect to $P$. 
\end{proof}

We are now ready to state and prove the main theorem of this section.

\begin{Thm}[Triple Normal Intersection Theorem]\label{TNIT}
Given a smooth ($C^2$) closed plane curve $\gamma$ and three angles $\theta_1, \theta_2, \theta_3$, such that $\theta_1 + \theta_2 + \theta_3 = 2\pi$ and $\theta_1, \theta_2, \theta_3 < \pi$, there exist three normals to $\gamma$ intersecting at a single point and forming angles $\theta_1, \theta_2,$ and $\theta_3$. 
\end{Thm}

\begin{proof}
Define $\tau_1 = \pi - \theta_1, \tau_2 = \pi - \theta_2$ and $\tau_3 = \pi - \theta_3$, and choose $T \in \mathcal{T}$ such that the angles of the vertices of the triangles in $T$ are $\tau_1$, $\tau_2$, and $\tau_3$ respectively. Because the curve $\gamma$ is $C^2$ smooth, there exists a triangle $\triangle DEF\in T$ of maximal area circumscribing $\gamma$, and there exist distinct points $A,B,C$ lying in the intersections of $\gamma$ and $\overline{FD},\overline{DE},\overline{EF}$, respectively. Furthermore, $\triangle DEF$ is also a triangle of maximal area in $T$ circumscribing $\triangle ABC$. Indeed, any triangle in the class $T$ circumscribing $\triangle ABC$ is contained inside (and no larger than) a triangle circumscribing $\gamma$ by moving the sides ``outward'' one by one. 

\begin{figure}[H]  
  \centering  
  \subfigure[Maximal triangle circumscribing $\gamma$]{\def\svgwidth{0.4\columnwidth} 
  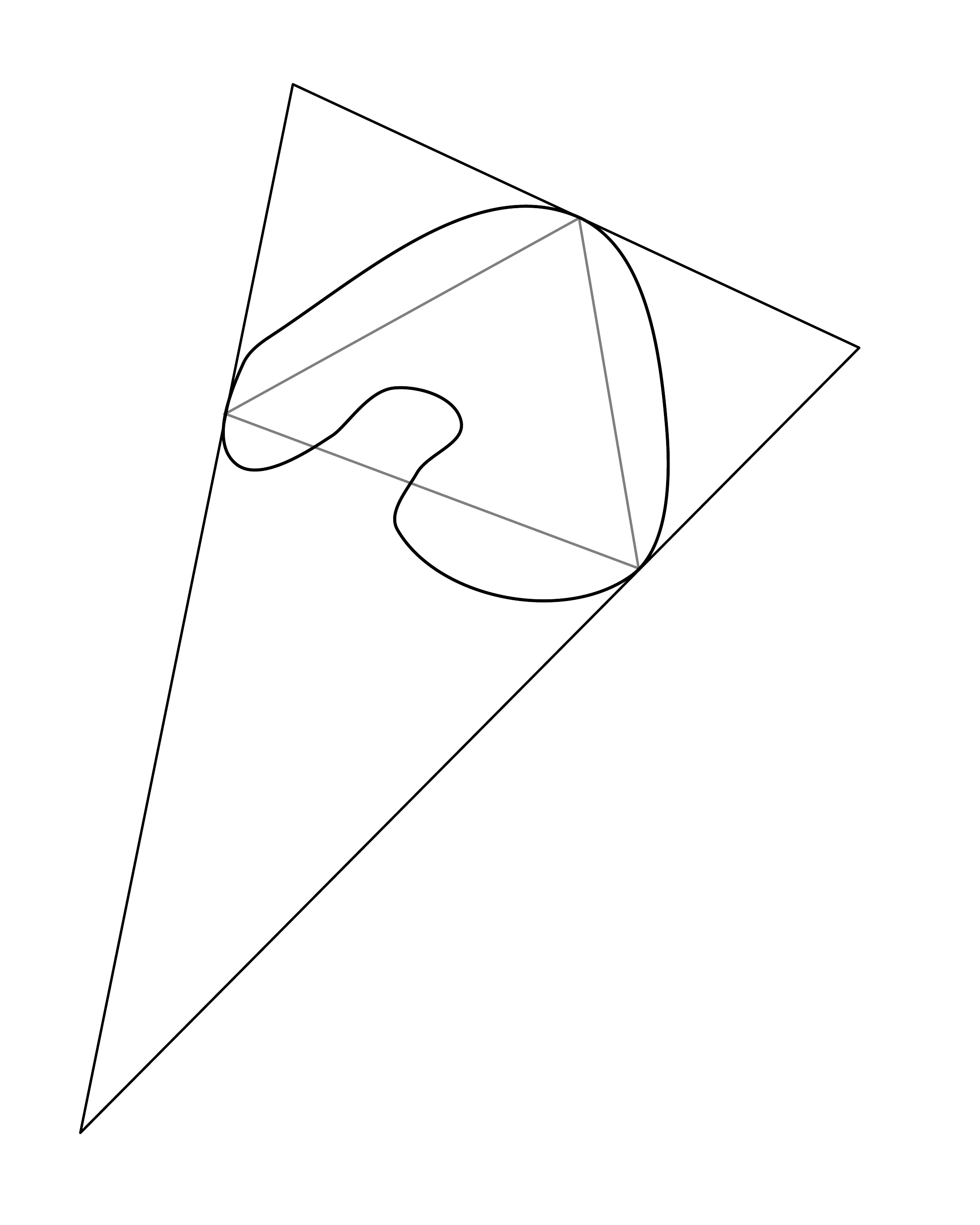}
  \quad
  \subfigure[Moving the sides ``outward'']{\def\svgwidth{0.4\columnwidth}  
  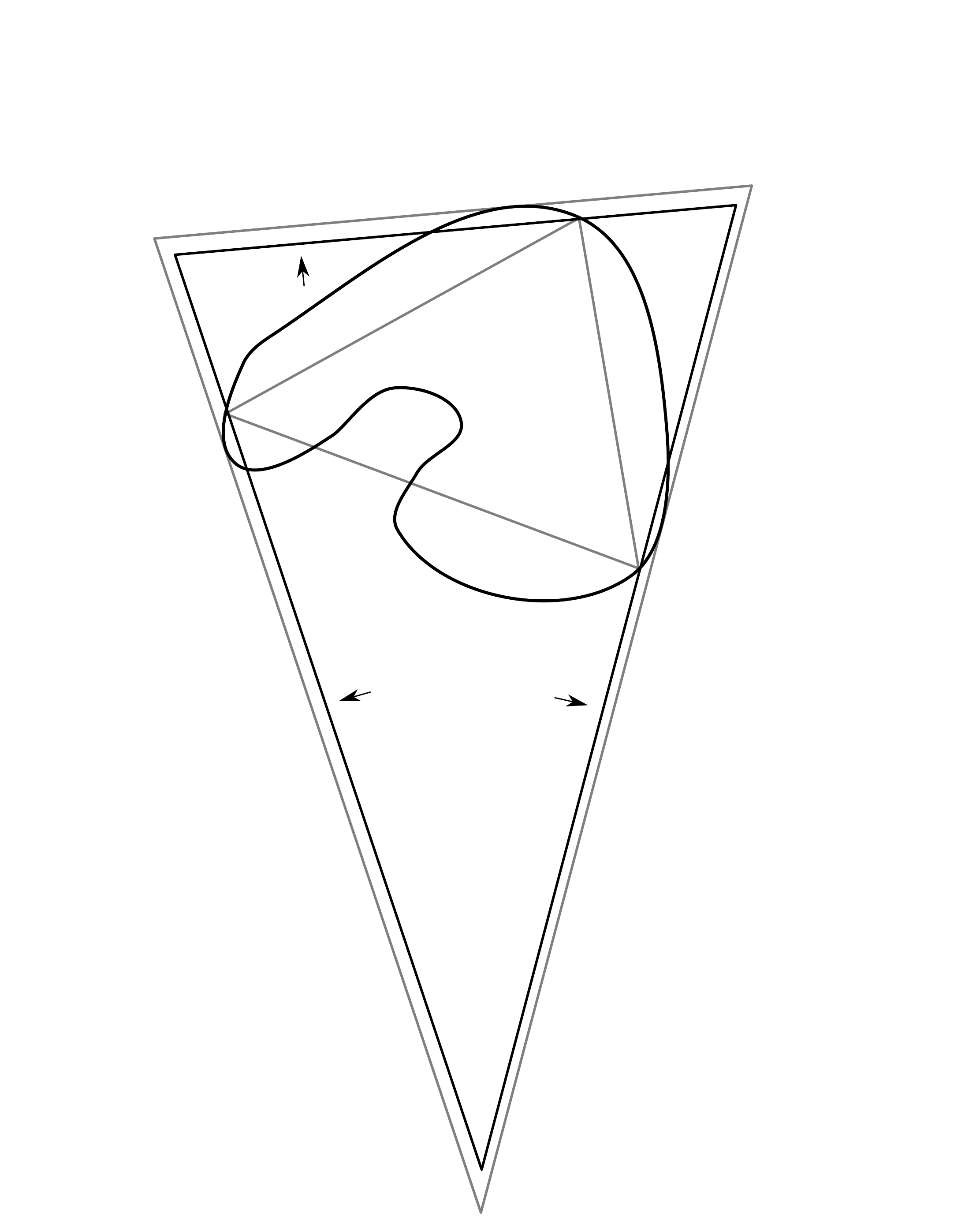}
  \caption{Maximal circumscribing triangles}
\end{figure}

Since $\triangle DEF\in T$ is the triangle of maximal area circumscribing $\triangle ABC$, it is antipedal to $\triangle ABC$ by Proposition \ref{antipedalprop}. So the three normal lines to $\gamma$ at $A,B,C$, all intersect at a point, forming the required angles.
\end{proof}

Setting $\theta_1=\theta_2=\theta_3=2\pi/3$ in Theorem \ref{TNIT} above gives us the following result.

\begin{Cor}\label{pln}
A smooth ($C^2$) closed plane curve has at least one tripod configuration. In particular, immersed plane curves with self intersection also possess at least one tripod configuration.
\end{Cor}


\section{Spherical and Hyperbolic Geometry: A Morse Theoretical Approach}\label{MTIntro}

We now extend our definition of tripod configurations to the spherical and hyperbolic geometries and again consider the question of which types of curves posess tripod configurations. Our strategy is to take a general curve and define a parameter space (a manifold with boundary) with a function defined on it; the pair is constructed so that the critical points of this function correspond to tripod configurations of the original curve. Using Morse theory for manifolds with boundary \cite{FL}, we then bound from below the number of critical points this function must possess on this parameter space, thus giving a lower bound the number of tripod configurations of a curve. Below, we define a natural extension of tripod configurations to general geometries, state our main results, and review the necessary Morse theory for the following two sections.

\begin{Def}
Given a $C^2$ closed curve $\gamma$, a tripod configuration of $\gamma$ consists of three geodesics normal to the curve, all coincident at a single point and pairwise making angles of $2\pi/3$.
\end{Def}

This is our original definition with geodesics replacing straight lines. We now state the following main result to be proven in Sections \ref{MTIntro} through \ref{CMT}.

\begin{Thm}\label{sphyp}
Every smooth ($C^2$) closed curve sufficiently close to a circle (excluding great circles on the sphere) in the spherical or hyperbolic geometry has at least two tripod configurations.
\end{Thm}

By sufficiently close to a circle we mean that the maximal diameter of the curve's evolute must be small in comparison to the minimal diameter of the curve. We exclude the case of curves close to a great circle in the spherical geometry since the necessary computations in Section \ref{CMI} are restricted points lying in a single hemisphere. This qualitative result gives the existence of a neighborhood around the circle in which smooth perturbations all contain 2 tripod configurations. We now introduce our parameter space and scalar function. 

\begin{Def}[Tripod Configuration Space]
Given a smooth closed curve $\gamma$ in a smooth $2$-dimensional manifold, let $R$ be the region enclosed by $\gamma$ and $\gamma_\epsilon$ be a parallel curve to $\gamma$ of constant distance $\epsilon$ away (in the ``outward'' direction, not in $R$). Then the \textit{tripod configuration space of} $\gamma$ is $\mathcal{P}_\gamma = \gamma_\epsilon \times \gamma_\epsilon \times \gamma_\epsilon \times R$. 
\end{Def}

In what follows we use the coordinates, $(t, u, v, p)$ where $t, u, v \in \gamma_\epsilon$ and $p \in R$, to discuss points in $\mathcal{P}_\gamma$. The region $R$ includes its boundary, and since $\gamma$ is a smooth curve, $\mathcal{P}_\gamma$ is a smooth manifold. 

\begin{Def}[Tripod Functional]
\textit{The Tripod Functional} of a curve, $\gamma$, is the function, $f : \mathcal{P}_\gamma \to \mathbb{R}$ defined by $(t, u, v, p) \mapsto \rho(t, p) + \rho(u, p) + \rho(v, p)$ where $\rho$ is the distance function on the ambient manifold. 
\end{Def}

Note that $f$ is a smooth function except possibly where $t, u,$ or $v$ coincides with $p$. But since the region $R$ is contained properly within $\gamma_\epsilon$, these points do not exist in our domain, and thus $f$ is smooth. For generic curves, this functional is Morse, i.e. has a non-singular Hessian. Below we establish that certain equivalence classes of its critical points from the interior of $\mathcal{P}_\gamma$ correspond to the tripod configurations of $\gamma$.

\begin{Prop}\label{critprop}
Let $\mathcal{C}$ be the set of interior critical points of $f$, and let $(t, u, v, p) \sim (x, y, z, p)$ if $\sigma(t, u, v) = (x, y, z)$ for some permutation on three objects $\sigma\in S_3$. Then for every element of $\mathcal{C} / \sim$, $\gamma$ has at least one tripod configuration.
\end{Prop}

\begin{proof}
The functional $f$ has an interior critical point at $(t_0, u_0, v_0, p_0)$ precisely when $\grad f = 0$ at that point, which implies:

\begin{align*}
\frac{d}{dt}|_{t_0} f &= \frac{d}{dt}|_{t_0} \rho(t, p_0) = 0,\\ 
\frac{d}{du}|_{u_0} f &= \frac{d}{du}|_{u_0} \rho(u, p_0) = 0, \\ 
\frac{d}{dv}|_{v_0} f &= \frac{d}{dv}|_{v_0} \rho(v, p_0) = 0. \\
\end{align*}

So  $t_0, u_0,$ and $v_0$ are critical points of the function $x \mapsto \rho(x, p)$. Thus, the (arc length minimizing) geodesic segments $\overline{t_0 p_0}$, $\overline{u_0 p_0}$, and $\overline{v_0 p_0}$ are orthogonal to the curve $\gamma$, and pairwise form angles of $2\pi/3$. This is true since in general $\frac{d}{dy} \rho(x, y)$ gives the unit vector pointing from $y$ to $x$, and

\[ 0 = \frac{d}{dp}|_{p_0} f = \frac{d}{dp}|_{p_0} \rho(t_0, p) + \frac{d}{dp}|_{p_0} \rho(u_0, p) + \frac{d}{dp}|_{p_0} \rho(v_0, p). \]

So the geodesic lines through $\overline{t_0 p}$, $\overline{u_0 p}$, and $\overline{v_0 p}$ form a tripod configuration of $\gamma$. Because we can permute the first three coordinates of our configuration space in six ways, there are exactly six critical points of the functional $f$ corresponding to a single tripod configuration. 
\end{proof}

\begin{Rem}
The critical points described above detect tripod configurations with tripod points \textit{inside} the curve only; tripod configurations as in Figure \ref{def3} with tripod points occuring outside of the curve will not be counted.
\end{Rem}

Morse theory for a functional $f$ on a manifold $M$ with boundary is concerned with the critical points of $f$ in the interior of $M$ and the critical points of $f$ when restricted to $\partial M$. In our situation, the functional $f$ has critical points in the interior of $\mc{P}_\gamma$ whenever $\grad f$ is zero and has critical points when restricted to $\partial \mc{P}_\gamma$ whenever $\grad f$ points either outwards or inwards orthogonally to $\mc{P}_\gamma$ from $\partial \mc{P}_\gamma$. The first situation corresponds to tripod configurations of $\gamma$ as discussed in Proposition \ref{critprop}. Using the notation of Laudenbach \cite{FL}, the last two situations correspond to Dirichlet or $D$ type critical points, and Neumann or $N$ type critical points, respectively; a critical point is said to be type $D$ if the gradient vector points orthogonally outward along the boundary, and type $N$ if the gradient vector points orthogonally inward along the boundary. Letting $n(p)$ be the outward pointing normal at the boundary point $p$, this condition may equivalently be formulated as $(\grad f|_p, n(p)) > 0$ for type $D$ critical points, and $(\grad f|_p, n(p)) < 0$ for type $N$ critical points. 

The \textit{Morse index} of a critical point denotes the number of negative eigenvalues of the Hessian $Hess(f)$ at that point. Following \cite{FL}, given a manifold with boundary $M$, we fix the following notation:

\begin{description}
\item[$C_k$] denotes the set of critical points of $f:int(M)\rightarrow\mb{R}$ of index $k$.
\item[$N_k$] denotes the set of critical points of $f:\partial M\rightarrow\mb{R}$ of type $N$ and index $k$.
\item[$D_k$] denotes the set of critical points of $f:\partial M\rightarrow\mb{R}$ of type $D$ and index $k-1$.
\item[$|\cdot|$] denotes the cardinality of the indicated finite set.
\end{description}

We define the Morse polynomials $\mc{M}_f^N$ and $\mc{M}_f^D$ as follows:

\begin{align*}
\mc{M}_f^N(T)&=\sum_k|C_k\cup N_k| T^k,
\\
\mc{M}_f^D(T)&=\sum_k|C_k\cup D_k| T^k.
\end{align*}

We define $\mc{P}_M$, the Poincar\'{e} polynomial of $M$:

\begin{align*}
\mc{P}_M(T)=\sum_k\text{rank}\ H_k(M;\mb{Z})\ T^k.
\end{align*}

We then have the following theorem from \cite{FL}:

\begin{Thm}[Laudenbach]\label{MT}
We have
\begin{align*}
\mc{M}_f^N(T)-\mc{P}_M(T)&=(1+T)Q^N(T),
\\
\mc{M}_f^D(T)-T^n\mc{P}_M(1/T)&=(1+T)Q^D(T),
\end{align*}
where $Q^N(T)$ and $Q^D(T)$ are polynomials with nonnegative coefficients, and $n$ is the dimension of the manifold $M$.
\end{Thm}

\begin{figure}[H]  
  \centering  
  \def\svgwidth{0.3\columnwidth}  
  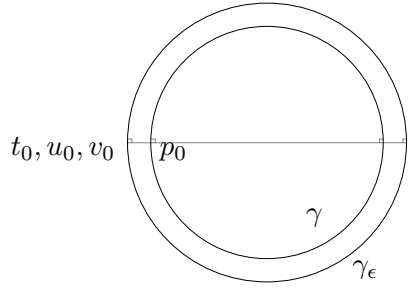
  \caption{A type $N$ critical point}
  \label{densesubmanifold}
\end{figure}

To study the number of critical points our tripod functional possesses in the interior of the tripod configuration space using Theorem \ref{MT}, we will analyze its type $D$ critical points. We make this choice since our configuration space may possess infinitely many type $N$ critical points along the boundary. Indeed, in Figure \ref{densesubmanifold}, when $t_0 = u_0 = v_0$ and $p_0$ is the closest point in $R$ to $t_0 = u_0 = v_0$, then the direction of greatest increase for $p\mapsto d(t_0, p) + d(u_0, p) + d(v_0, p)$ is directly into $R$, normal to $\gamma$. So $(t_0, u_0, v_0, p_0)$ is a type $N$ critical point, and $M_N = \{ (t, t, t, p) : p \text{ is the closest point to }t \in \gamma \}$ is a submanifold of type $N$ critical points.

\section{Type $D$ Critical Points}

Our goal in this section is to describe when type $D$ critical points occur for the functional $f:P_\gamma\rightarrow\mb{R}$. Recall the notation fixed in the previous section. The functional $f$ has a boundary critical point of type $D$ at $(t_0,u_0,v_0,p_0)$ if and only if the gradient vector of $f$ points orthogonally outward along the boundary of $P_\gamma$. Equivalently, this requires that line segments $\overline{t_0 p_0},\overline{u_0 p_0},\overline{v_0 p_0}$ are orthogonal to $\gamma_\epsilon$ (and thus $\gamma$), that $p_0$ lies on $\gamma$, and that the vector $d/dp|_{p_0} f(t_0,u_0,v_0,p)$ in the $2$-dimensional space containing $\gamma$ is normal $\gamma$, pointing outwards. We therefore consider the possible numbers of distinct lines normal to $\gamma_\epsilon$ all passing through a single point $p$ on $\gamma$.

\begin{figure}[H]  
  \centering  
  \def\svgwidth{0.5\columnwidth}  
  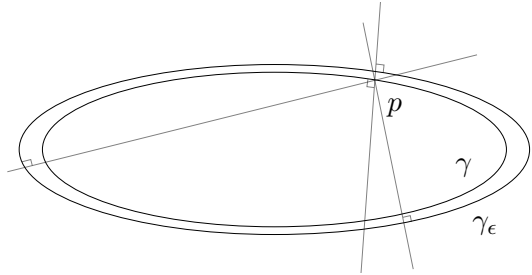
  \caption{Three lines through $p$ normal to $\gamma_\epsilon$}
  \label{normalepsilon}
\end{figure} 

As discussed earlier, we assume in Sections \ref{MTIntro} through \ref{CMT} that $\gamma$ is sufficiently close to a circle, so that in particular $\gamma$ encloses its evolute. We will next see that this is sufficient to ensure that there are at most two distinct lines normal to $\gamma_\epsilon$ passing through a single point on $\gamma$. First, recall that the evolute of a smooth curve is the envelope of its normal lines; in particular, the evolute of a circle degenerates to a single point, its center. 

\begin{Lem}
Let $\gamma$ be a smooth closed curve sufficiently close to a circle, so that its evolute lies strictly inside $\gamma$. Fix $\epsilon>0$; then for every point of $\gamma$ there exist exactly two lines passing through it which are also normal to $\gamma_\epsilon$. 
\end{Lem}

\begin{proof}
In general, given a fixed curve in a $2$-dimensional space, we may define a function on the space by mapping each point in the space to the number of distinct lines normal to the curve passing through that point. This number is constant for points in the connected components of the complement of the evolute of the curve (see, for instance, \cite{Mathomni}). Now $\gamma$ is obtained from a circle by a sufficiently small deformation so that the common evolute of $\gamma$ and $\gamma_\epsilon$ does not intersect $\gamma$, so that $\gamma$ and $\gamma_\epsilon$ lie in the same connected component of the complement of the evolute. Thus the number of lines normal to $\gamma_\epsilon$ passing through a point on $\gamma$ is always two.
\end{proof}

We therefore see that if $(t_0,u_0,v_0,p_0)$ is a type $D$ critical point of $f$, then with $p_0$ fixed, each of $t_0,u_0,v_0$ are one of exactly two points on $\gamma_\epsilon$ whose line segments connecting them to $p_0$ are normal to $\gamma$. Because $d/dp|_{p_0} f(t_0,u_0,v_0,p)$ is normal to $\gamma$ pointing in the outward direction, it must also be the case that $t_0,u_0,v_0,p_0$ all lie on a single diameter of $\gamma$. Recall that a diameter of convex closed curve $\gamma$ is a line normal to the curve at two points. Finally, since $d/dp|_{p_0} f$ should point outwards from $\gamma$, we conclude that all type $D$ critical points are associated with diameters of $\gamma$ in one of the two configurations shown in Figure \ref{typeDcases}.

\begin{figure}[H]
\centering  
  \subfigure[Case $1$]{\def\svgwidth{0.4\columnwidth}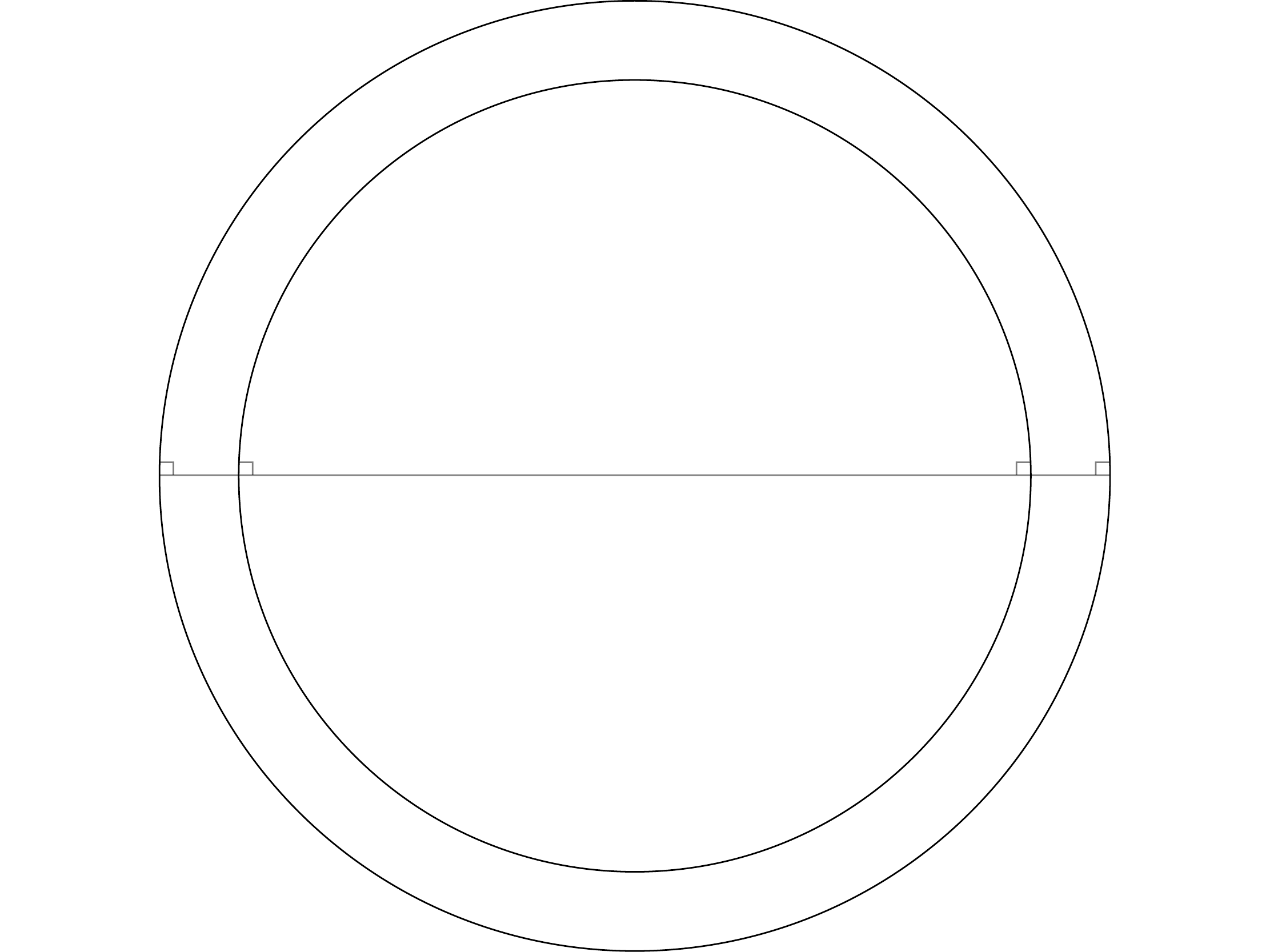}
  \qquad\qquad
  \subfigure[Case $2$]{\def\svgwidth{0.4\columnwidth}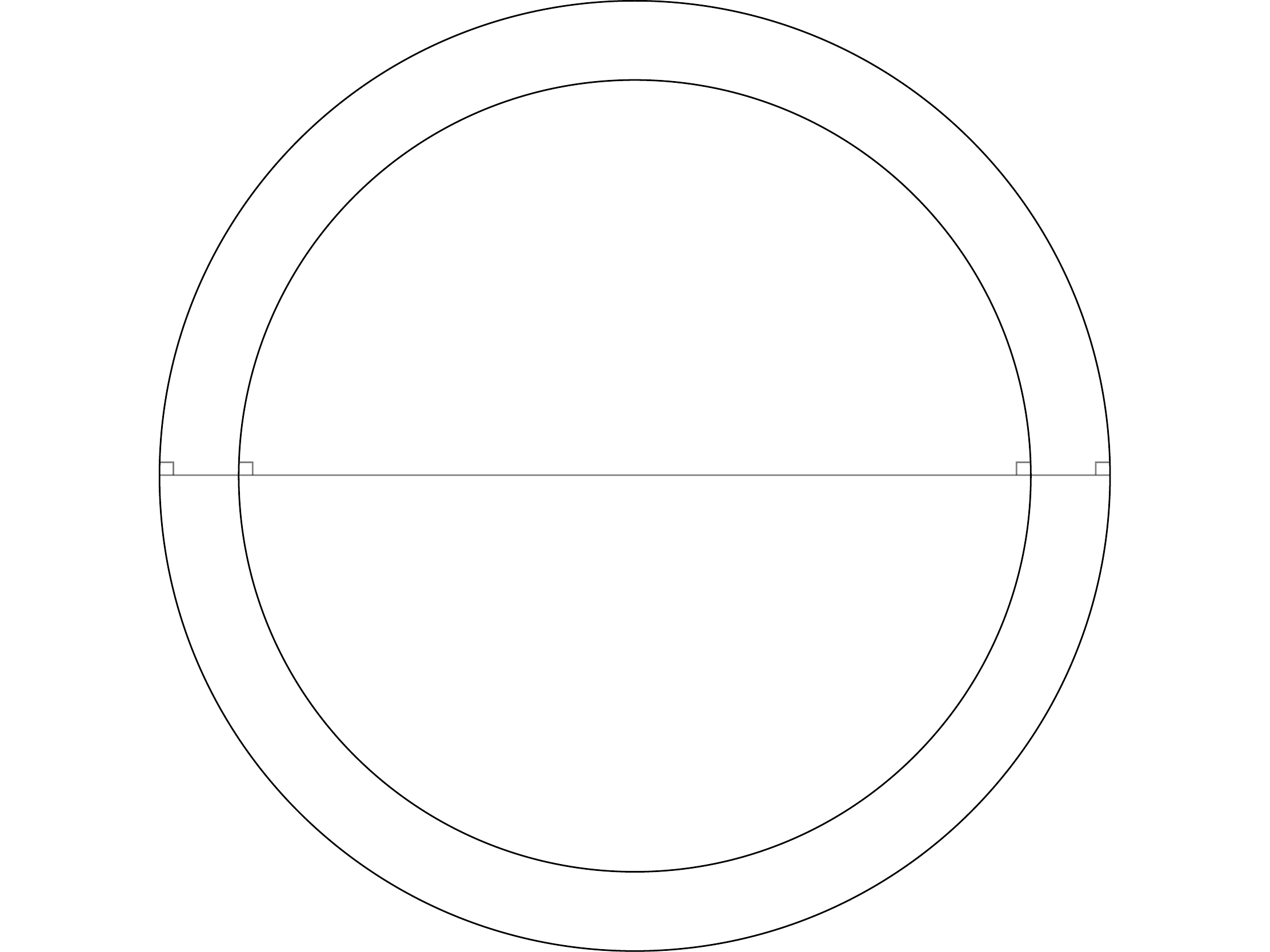}
\caption{The only possible configurations of type $D$ critical points}
\label{typeDcases}
\end{figure}

\section{Computation of Morse Indices}\label{CMI}

We proceed to find of the Morse indices of type $D$ critical points of the tripod functional $f$ defined from a curve $\gamma$ in the planar, spherical or hypberbolic geometries in the Cases $1$ and $2$ shown in Figure \ref{typeDcases} by computing the indices of $Hess(f)$ at these critical points. To do this, we approximate $\gamma$ and $\gamma_\epsilon$ up to second order by osculating circles near the points $t_0,u_0,v_0,p_0$. In our calculations the condition that $\gamma$ is sufficiently close to a circle is used to assume that the radii of the osculating circles are arbitrarily large in comparison to the distance between their centers, and that the radii are approximately equal.

In fact the indices of the type $D$ critical points are the same in the planar, spherical, and hyperbolic geometries. We first state the following definition before giving the results of our computations.

\begin{Def}[Orientation of a Diameter]\label{diamdef}
If $\overline{ab}$ is a diameter of the smooth curve $\gamma$ and if $c(x)$ is the center of curvature of $\gamma$ at $x$, then the orientation of $\overline{ab}$ is the dot product of the unit vector pointing from $a$ to $b$ and the unit vector pointing from $c(a)$ to $c(b)$.
\end{Def}

\begin{figure}[H]
\centering  
  \subfigure[Positively oriented diameter]{\def\svgwidth{0.4\columnwidth}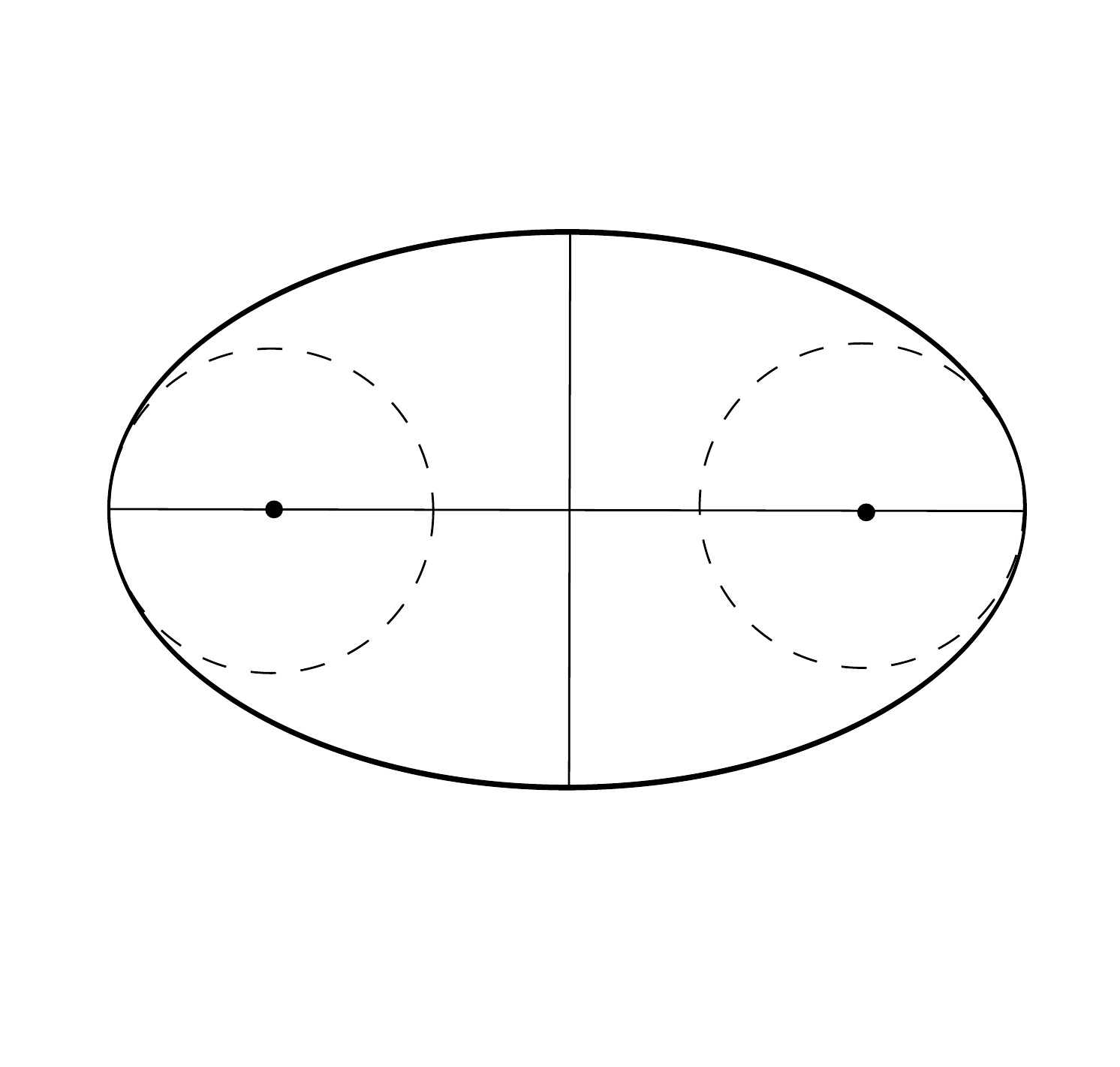}
  \quad
  \subfigure[Negatively oriented diameter]{\def\svgwidth{0.4\columnwidth}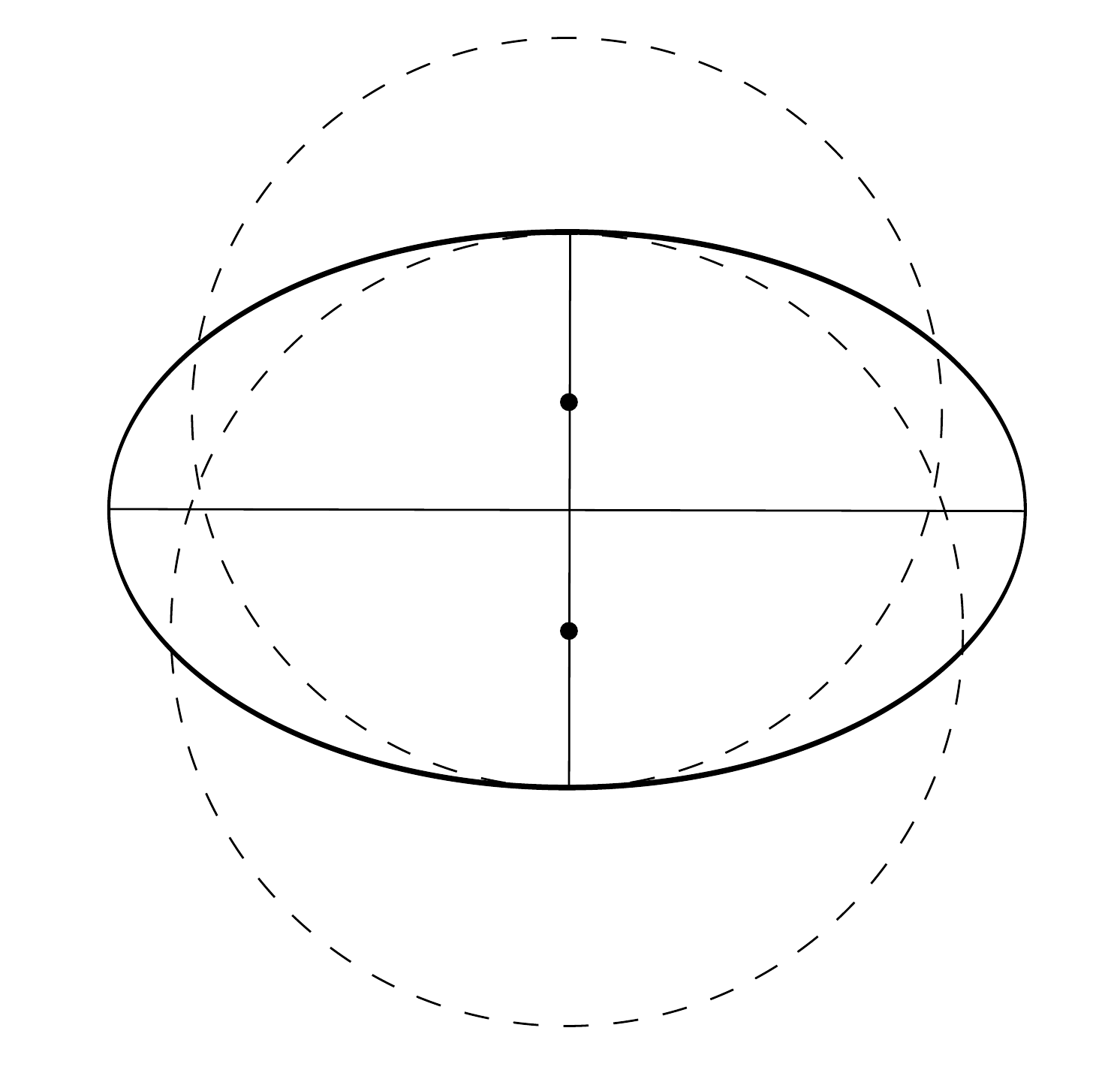}
  \caption{Orientations of diameters}
\end{figure}

The results of our computations of the indices of type $D$ critical points are then as follows, labeled by the configurations depicted in Figure \ref{typeDcases}:

\begin{align*}
\text{Case }1 \;\;
\begin{cases}
3,& \text{ for negatively oriented diameters},
\\
4,& \text{ for positively oriented diameters}.
\end{cases}
\\
\text{Case }2 \;\;
\begin{cases}
2,& \text{ for negatively oriented diameters},
\\
3,& \text{ for positively oriented diameters}.
\end{cases}
\end{align*}

The computations in the planar, spherical, and hyperbolic geometries are quite similar. Below we include some of the details of our computations in the hyperbolic geometry setting.

\subsection{Case $1$, Hyperbolic Geometry}

We use the Poincar\'{e} disk model shown in Figure \ref{hyperbolicfig1}; $\overline{oq}$ is a segment of a diameter of the curve $\gamma$ (not shown) so that the type $D$ critical point $(t_0,u_0,v_0,p_0)$ of $f$ is given by $p_0$ lying on this diameter and $\gamma$, closer to $o$, and $t_0=u_0=v_0$ all lying on the opposite side of this diameter on $\gamma_\epsilon$. The curve $\gamma$ has radius of curvature $r$ at point $p_0$, with center of curvature $o$, while $\gamma_\epsilon$ has radius of curvature $R$ at point $t_0=u_0=v_0$, with center of curvature $q$. Note carefully that $||\overline{oq}||=d$ is defined to be a signed distance with sign corresponding to the orientation of the diameter of $\gamma$ through $\overline{oq}$. Our assumption that $\gamma$ is sufficiently close to a circle allows us to assume that $r$ is close to $R$ and that the magnitude of $d$ is small.

\begin{figure}[H]
	\centering
	\def\svgwidth{0.8\columnwidth}
	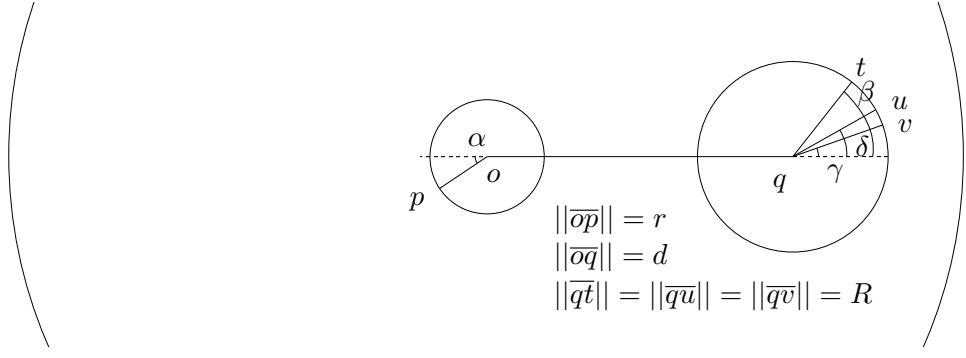
	\caption{Case $1$, hyperbolic geometry}
	\label{hyperbolicfig1}
\end{figure}

We perturb $t,u,v,p$ from $t_0,u_0,v_0,p_0$, respectively, along the corresponding curves ($\gamma_\epsilon$ and $\gamma$), approximating up to second order by moving along the appropriate osculating circles by angles $\alpha,\beta,\gamma,\delta$, giving the following coordinates:

\begin{align*}
p&=(-r\cos\alpha,-rsin\alpha)
\\
&\approx(-r(1-\frac{\alpha^2}{2}),-r\alpha),
\\
t&=(d+R\cos\beta,R\sin\beta)
\\
&\approx(d+R(1-\frac{\beta^2}{2}),R\beta),
\\
u&=(d+R\cos\gamma,R\sin\gamma)
\\
&\approx(d+R(1-\frac{\gamma^2}{2}),R\gamma),
\\
v&=(d+R\cos\delta,R\sin\delta)
\\
&\approx(d+R(1-\frac{\delta^2}{2}),R\delta).
\end{align*}

Define the function
\begin{align*}
g(\alpha,\beta,\gamma,\delta)=d(t,p)+d(u,p)+d(v,p),
\end{align*}
where $d(x,y)=\text{arccosh}\left(1+2\frac{||x-y||^2}{(1-||x||^2)(1-||y||^2)}\right)$ is the hyperbolic metric and $||\cdot||$ is the usual metric in the plane restricted to the disk. We then analyze the signs of the principal minors of the Hessian of $g$. Below, $M_i$ denotes the $i$th principal minor of the $4\times 4$ matrix $Hess(g)$, the determinant of the $i\times i$ upper left corner of $M$.
\begin{enumerate}
\item $M_4$: Letting $r=R$, we find that
\begin{align*}
\lim_{d\rightarrow 0}\frac{\det(M_4)}{d}={\frac {{-6R}^{3}}{ \left( {R}^{4}+1+2\,{R}^{2} \right)  \left( -1+{
R}^{4} \right) }}.
\end{align*}
\item $M_3$: Letting $r=R$ and $d=0$ we find that
\begin{align*}
\det(M_3)=-\frac{R^3}{(R^2+1)^3}.
\end{align*}
\item $M_2$: Letting $r=R$ and $d=0$ we find that
\begin{align*}
\det(M_2)=\frac{2R^2}{(R^2+1)^2}.
\end{align*}
\item $M_1$: Letting $r=R$ and $d=0$ we find that
\begin{align*}
\det(M_1)=\frac{-3R}{R^2+1}.
\end{align*}
\end{enumerate}

Having assumed that $d$ is small, we obtain:

\begin{tabular}{l l}
	Leading Minor & Sign \\
	$M_1$ & $-$ \\
	$M_2$ & $+$ \\
	$M_3$ & $-$ \\
	$M_4$ & $\begin{cases} - \text{ if } d<0 \\
						   + \text{ if } d>0. \\
			 \end{cases}$ 
	\end{tabular}

The following property of linear algebra \cite{kaplansky} then allows us to conclude that Morse index of type $D$ critical points of the tripod functional $f$ in the Case $1$ configuration is $3$ if $d<0$ and $4$ if $d>0$.

\begin{Prop}\label{linalgprop}
Let $A$ be an $n\times n$ symmetric matrix with principal minors $A_1,A_2,\ldots,A_n$ nonzero. Then $A_1, A_2/A_1,\ldots,A_n/A_{n-1}$ are the diagonal entries in a diagonalization of $A$. 
\end{Prop}

\subsection{Case $2$, Hyperbolic Geometry}

\begin{figure}[H]
	\centering
	\def\svgwidth{0.8\columnwidth}
	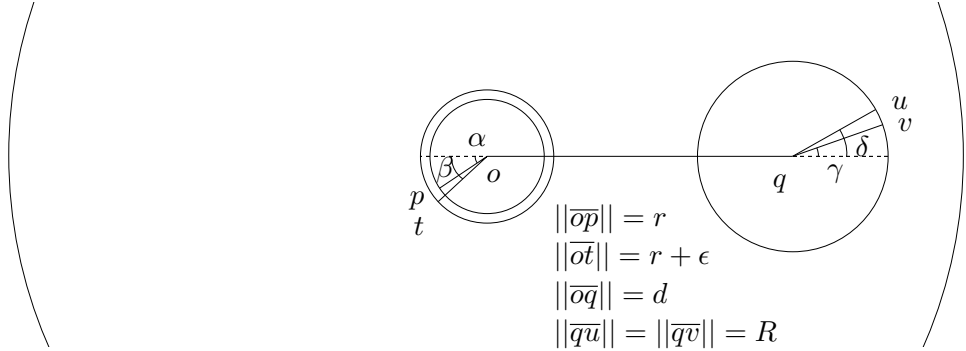
	\caption{Case $2$, hyperbolic geometry}
	\label{hyperbolicfig}
\end{figure}

We have nearly the same situation as before, but now $t_0$ lies on $\gamma_\epsilon$ on the same side of the diameter through $\overline{oq}$ as $p$ on $\gamma$. Using approximations similar to before, we have:
\begin{align*}
p&=(-r\cos\alpha,-r\sin\alpha),
\\
t&=(-(r+\epsilon)\cos\alpha,-(r+\epsilon)\sin\alpha),
\\
u&=(d+R\cos\gamma,R\sin\gamma),
\\
v&=(d+R\cos\delta,R\sin\delta).
\end{align*}
Again we define the function
\begin{align*}
g(\alpha,\beta,\gamma,\delta)=d(P,X)+d(P,Y)+d(P,Z),
\end{align*}
where $d(x,y)=\text{arccosh}\left(1+2\frac{||x-y||^2}{(1-||x||^2)(1-||y||^2)}\right)$ is the hyperbolic metric and $||\cdot||$ is the usual metric in the plane restricted to the disk. We again analyze the signs of the principal minors of $Hess(g)$, and in addition to our assumptions that $\gamma$ is close to a circle we may further take $\epsilon$ to be arbitrarily small.
\begin{enumerate}
\item $M_4$: Letting $r=R$, we find that
\begin{align*}
\lim_{\epsilon\rightarrow 0^+}\lim_{d\rightarrow 0}\frac{\epsilon}{d}\det(M_4)=-\frac{2R^4(R^2+1)}{(R^4+1)^{\frac{3}{2}}(R^2-1)^2}.
\end{align*}
\item $M_3$: Letting $r=R$ and $d=0$, we find that
\begin{align*}
\lim_{\epsilon\rightarrow 0^+}\det(M_3)=\frac{2R^4}{(1-R^2)(R^2+1)^2}.
\end{align*}
\item $M_2$: Letting $r=R$ and $d=0$, we find that
\begin{align*}
\lim_{\epsilon\rightarrow 0^+}\det(M_2)=-\frac{4R^3}{1-R^4}.
\end{align*}
\item $M_1$: Letting $r=R$ and $d=0$, we find that
\begin{align*}
\lim_{\epsilon\rightarrow 0^+}\det(M_1)=\frac{2R^2}{1-R^2}.
\end{align*}
\end{enumerate}
Again applying Proposition \ref{linalgprop}, we find that the Morse index of type $D$ critical points of the tripod functional $f$ in the Case $2$ configuration is $2$ if $d<0$ and $3$ if $d>0$.

\section{Conclusions from Morse Theory}\label{CMT}

In the previous section, we computed the Morse indices of type $D$ critical points of the tripod functional $f$ of a curve $\gamma$ sufficiently close to a circle along the boundary of our tripod configuration space in the planar, spherical, and hyperbolic geometries. With this information we may prove our results on tripod configurations. 

First we note that the diameters of a convex curve come in pairs of positively and negatively oriented diameters as defined in Definition \ref{diamdef}. This can be shown using Morse theory to study the distance function defined on pairs of points on the curve, similar to the approach employed in \cite{Halpern}. Diameters of a curve $\gamma$ also coincide with $2$-periodic billiard trajectories inside $\gamma$; see \cite{STB} for a discussion of signs of diameters in terms of the stability of $2$-periodic billiard trajectories.

\begin{proof}[Proof of theorem \ref{sphyp}]
Let $\gamma$ be a closed smooth curve in either the plane, spherical, or hyperbolic geometry. Let the number of diameters of $\gamma$ be $d$, and let $n = \frac{d}{2}$. Thus, $n$ gives both the number of positively oriented diameters and the number of negatively oriented diameters of $\gamma$. Now given a critical point $(t_0,u_0,v_0,p_0)$, we may either permute $t_0,u_0,v_0$ or move each of $t_0,u_0,v_0,p_0$ to the opposite side of the diameter associated to the critical point. Therefore for each diameter of $\gamma$ there exist $2$ type $D$ critical points in the Case $1$ configuration, and $6$ type $D$ critical points in the Case $2$ configuration. Using the Morse indices determined by our computations in Section \ref{CMI}, we see that the Morse polynomial for the type $D$ critical points of $\gamma$ is:

\[ \mc{M}_f^D(t) = C(t) + n(2t^4 + 6t^3) + n(2t^5 + 6t^4), \]

where $C(t)$ is the polynomial

\[ C(t) = \sum_k |C_k| t^k. \]

The Poincar\'{e} polynomial of the tripod configuration space is:

\[ P_M(t) = (1 + t)^3. \]

Thus, by Proposition \ref{MT} we have:

\begin{align*}
M_f^D(t) - t^5 P_M(\frac{1}{t}) &= (1 + t)Q^D(t), \\
C(t) + n(6t^3 + 8t^4 + 2t^5) - t^5(1 + \frac{1}{t})^3 &= (1 + t)Q^D(t), \\
C(t)+(1+t)(2nt^4+6nt^3)-(1+t)^3t^2 &=(1 + t)Q^D(t).
\end{align*}

This shows that $(1 + t)$ divides $C(t)$. Further note that
\begin{align*}
(1+t)(2nt^4+6nt^3)-(1+t)^3t^2=(2n-1)t^5+(8n-3)t^4+(6n-3)t^3-t^2.
\end{align*}
Now the $t^2$ coefficient above is $-1$, while $Q^D(t)$ has nonnegative coefficients, so $(1+t)(2nt^4+6nt^3)-(1+t)^3t^2\neq (1+t)Q^D(t)$ and thus $C(t)\not\equiv 0$.
Let $C(t)=(1+t)(a_k t^k+\cdot+a_j t^j)$, where $0\leq j\leq k$ and $a_k,a_j\neq 0$. Because $C(t)=a_k t^{k+1}+\cdots+a_j t^j$ has nonnegative coefficients, we see that $a_k$ and $a_j$ must be strictly positive. It follows that $C(t)$ has at least two terms of different degree with positive coefficients; i.e. $f$ has at least two critical points corresponding to two distinct tripod configurations.
\end{proof}

We conclude this section by stating two conjectures below which would generalize Theorem \ref{sphyp} and appear to be natural extensions of results in the planar case.

\begin{Conj}\label{sphypconj}
Every smooth closed convex curve in the spherical or hyperbolic geometry has at least two tripod configurations.
\end{Conj}

\begin{Conj}\label{tripodconfigconj}
Every smooth closed curve in the spherical or hyperbolic geometry has at least one tripod configuration.
\end{Conj}

\section{Tripod configurations for regular polygons}\label{poly}

In this section we discuss an extension of the problem of counting tripod configurations to the setting of regular polygons. Recall that for a triangle with no angles exceeding $2\pi/3$, there exists a unique point inside the triangle at which the lines drawn from that point to the triangle vertices make angles of $2\pi/3$, the Fermat-Toricelli point. Given a polygon, define a tripod configuration to be three lines $\ell_1,\ell_2,\ell_3$ passing through a point $p$ such that each of the lines passes through a different vertex of the polygon and is perpendicular to a support line of the polygon through that vertex. We consider below whether such configurations exist for regular polygons.

In general, if a tripod configuration exists for any polygon with lines passing through vertices $v_1,v_2,v_3$, then the point $p$ where all three lines coincide must be the Fermat point of the triangle formed by these three vertices (which must also have no angles exceeding $2\pi/3$); however, the additional conditions that $\overline{pv_1}$ must make an angle less than $\pi/2$ with the two sides meeting vertex $v_1$ and analogously for $p$ and $v_2,v_3$ must also be satisfied.

\begin{figure}[H]
	\begin{tikzpicture}[scale=0.5]
		\draw [thick] (0,0) -- (-3, 4) -- (0, 8) -- (5, 8) -- (8, 4) -- (5, 0) -- (0, 0);
		\draw [<->] (-4, 0) -- (-2, 8);
		\draw [dashed, ->] (-3, 4) -- (1, 3);
	\end{tikzpicture}
	\caption{A support line at a vertex of a polygon}
\end{figure}
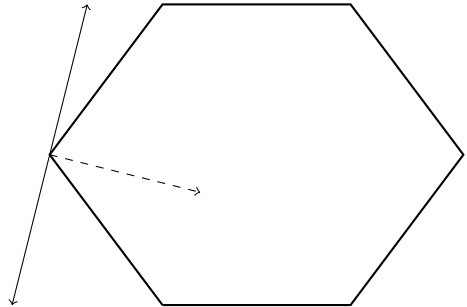

The remainder of this section is devoted to proving the following theorem:

\begin{Thm}\label{regpoly}
A regular $n$-vertex polygon has $n$ tripod configurations if $3\nmid n$ and has $n/3$ tripod configurations if $3\mid n$.
\end{Thm}

\subsection{There exist tripod configurations for all regular polygons}

We know that there exists a single tripod configuration for an equilateral triangle corresponding to its Fermat-Toricelli point. Now consider a regular polygon $Q$ with $n$ sides and vertices labeled $v_0,\ldots,v_{n-1}$ (in cyclic order). We consider candidate ``isoceles'' tripod configurations: we choose three vertices of $Q$ that make an isoceles triangle, find its Fermat-Toricelli point, and check whether the three lines passing through it and one of the three chosen vertices form a tripod configuration of $Q$ by determining whether support line condition is satisfied. By symmetry it suffices to consider the isoceles triangles with vertices $v_0, v_k,v_{n-k}$ and Fermat-Toricelli point $P$. Then we compute (working in degrees):

\begin{align*}
&\text{Vertex angles of }Q\quad y:=180-\frac{360}{n}
\\
a:=&\measuredangle v_1 v_0 v_k=\measuredangle v_{k-1}v_kv_0=\measuredangle v_{n-1}v_0v_{n-k}=\measuredangle v_{n-k+1}v_{n-k}v_0=\frac{(180-y)(k-1)}{2}
\\
b:=&\measuredangle v_{k+1}v_kv_{n-k}=\measuredangle v_{n-k-1}v_{n-k}v_k=\frac{(180-y)(n-2k-1)}{2}
\end{align*}

The support line condition described above is equivalent to $\measuredangle v_{k+1}v_kP<90$ and $\measuredangle v_{k-1}v_kP<90$. Using the above expressions, we find that
\begin{align*}
\measuredangle v_{k+1}v_kP=30+b=90+\frac{120n-360k-180}{n}
\\
\measuredangle v_{k-1}v_kP=y-(30+b)=90+\frac{-120n+360k-180}{n}
\end{align*}
So we require that $|120n-360k|<180$. There are the three cases $n=3m$, $n=3m+1$, and $n=3m+2$ for some $m\in\mb{Z}_+$. If $n=3m$ or $n=3m+1$, then only $k=m$ satisfies this condition; if $n=3m+2$, then only $k=m+1$ satisfies this condition.

\subsection{The tripod configurations for regular polygons listed above are the only tripod configurations}

As for tripod configurations for smooth curves, the support lines corresponding to the lines forming a tripod configuration of a regular polygon $Q$ form an equilateral triangle; any tripod configuration of a polygon corresponding to vertices $v_{i_1}, v_{i_2}, v_{i_3}$ is associated to an equilateral triangle enclosing $Q$ and meeting it at the three vertices. 

We may count the configurations of ``circumscribing'' equilateral triangles about $Q$. By symmetry it suffices to count the number of such triangles passing through a particular point, say $v_0$, when the vertices of regular $n$-polygon $Q$ are labeled cyclically as before, and we may further suppose that the angle made by the side of the circumscribing equilateral triangle passing through vertex $v_0$ measures less than $\frac{180-y}{2}$.

\begin{figure}[H]  
  \centering  
  \def\svgwidth{0.3\columnwidth}  
  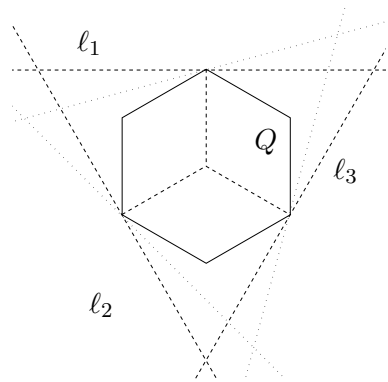
  \caption{Rotating ``circumscribing'' equilateral triangles about a regular polygon}  
\end{figure}

We again consider the three cases $n=3m$, $n=3m+1$, $n=3m+2$ separately. For $n=3m$, begin with the equilateral triangle circumscribed about $Q$ with sides which are segments on lines $\ell_1, \ell_2, \ell_3$ passing through vertices $v_0, v_m, v_{2m}$ respectively. Rotating $\ell_1$ about $v_0$ towards $\overline{v_0 v_1}$ decreases the angle between $\ell_1$ and $\overline{v_0 v_1}$ by the same amount that the angles between $\ell_2$ and $\overline{v_m v_{m+1}}$ as well as $\ell_3$ and $\overline{v_{2m}v_{2m+1}} $ decrease as $\ell_2$ and $\ell_3$ are rotated about $v_m$ and $v_{2m}$, respectively, to ensure that the triangle defined by $\ell_1,\ell_2,\ell_3$ remains equilateral. Continuing to rotate $\ell_1,\ell_2,\ell_3$ in this manner, we will find no new circumscribing configurations (up to rotational symmetry) will be produced once $v_0 v_1$ lies on $\ell_1$. So the only tripod configuration for $Q$ when $n=3m$ is the one associated with the triple described above: $v_0, v_m, v_{2m}$ and its rotated analogues.

Next we consider the case $n=3m+1$. From before we know there exists a circumscribing equilateral triangle about $Q$ with sides lying on lines $\ell_1, \ell_2, \ell_3$ passing through vertices $v_0, v_m, v_{2m+1}$ respectively. We again consider all possible circumscribing equilateral triangles by rotating $\ell_1, \ell_2, \ell_3$, with $\ell_1$ rotated about $v_0$ to decrease its angle with $\overline{v_0v_1}$ and $\ell_2$, $\ell_3$ rotated in the same direction and possibly translated in order that the equilateral triangle defined by $\ell_1, \ell_2, \ell_3$ continues to circumscribe $Q$. At any point of the rotation of $\ell_1$ towards $\overline{v_0v_1}$, $\ell_2$ will be rotating about $v_m$ or $v_{m+1}$, and $\ell_3$ will be rotating about $v_{2m+1}$ or $v_{2m+2}$. So we only need to check whether any of the three vertex triples $v_0,v_m,v_{2m+2}$, $v_0,v_{m+1},v_{2m+1}$, and $v_0,v_{m+1},v_{2m+2}$ are associated with tripod configurations. After rotation we see that $v_0,v_{m+1},v_{2m+1}$ is equivalent to $v_0,v_m,v_{2m+1}$, while $v_0,v_{m+1},v_{2m+2}$ is a distinct isoceles configuration; the previous section showed that this is not associated with a tripod configuration. It remains to consider $v_0,v_m,v_{2m+2}$. This occurs only if the acute angle between $\ell_2$ and $\overline{v_{m-1}v_m}$ is smaller than the acute angle between $\ell_2$ and $v_mv_{m+1}$. But it is easily computed that the two angles (in order) measure in degrees
\begin{align*}
\frac{300}{3m+1}
\quad\text{and}\quad
\frac{60}{3m+1}.
\end{align*}
So the last case is also not associated with a tripod configuration, and the only tripod configurations are the ones associated with the triple $v_0, v_m, v_{2m+1}$ and its rotated analogues.

Finally we consider the case $n=3m+2$. The argument goes as before, and we then need to consider the following triples: $v_0,v_{m+1},v_{2m+2}$, $v_0,v_{m+2},v_{2m+1}$, $v_0,v_{m+2},v_{2m+2}$. Now $v_0,v_{m+1},v_{2m+2}$ is equivalent to by symmetry to $v_0,v_{m+1},v_{2m+1}$, and 
$v_0,v_{m+2},v_{2m+2}$ corresponds to another isoceles case known not to be associated with a tripod configuration. Finally we consider $v_0,v_{m+2},v_{2m+1}$; if this triple were admissible, then the acute angle between $\ell_2$ and $\overline{v_{m-1}v_m}$ would be larger than the acute angle between $\ell_2$ and $\overline{v_mv_{m+1}}$. But in order, the two angles measure (in degrees)
\begin{align*}
\frac{60}{3m+2}
\quad\text{and}\quad
\frac{300}{3m+2}.
\end{align*}
So the last case also does not correspond to a tripod configuration. We conclude that for all regular polygons $Q$, the tripod configurations described above are the only tripod configurations of $Q$, with the exact counts arising from rotational symmetry.

\section{Acknowledgments}\label{ackn}

The authors of this paper are very pleased to acknowledge the generosity and mentorship of ICERM and Brown University for providing the space in which the research behind this paper could take place at ICERM's summer REU. We would like to thank Sergei Tabachnikov and Ryan Greene for being our mentors and the directors of our research, and Professor Richard Schwartz for his help and insight on the project as well. We would also like to thank Nakul Luthra for his assistance in our research.

\end{document}

%% file: fermat2.pdf_tex
\begingroup%
  \makeatletter%
  \providecommand\color[2][]{%
    \errmessage{(Inkscape) Color is used for the text in Inkscape, but the package 'color.sty' is not loaded}%
    \renewcommand\color[2][]{}%
  }%
  \providecommand\transparent[1]{%
    \errmessage{(Inkscape) Transparency is used (non-zero) for the text in Inkscape, but the package 'transparent.sty' is not loaded}%
    \renewcommand\transparent[1]{}%
  }%
  \providecommand\rotatebox[2]{#2}%
  \ifx\svgwidth\undefined%
    \setlength{\unitlength}{640bp}%
    \ifx\svgscale\undefined%
      \relax%
    \else%
      \setlength{\unitlength}{\unitlength * \real{\svgscale}}%
    \fi%
  \else%
    \setlength{\unitlength}{\svgwidth}%
  \fi%
  \global\let\svgwidth\undefined%
  \global\let\svgscale\undefined%
  \makeatother%
  \begin{picture}(1,0.8125)%
    \put(0,0){\includegraphics[width=\unitlength]{fermat2.pdf}}%
    \put(0.47603436,0.34090881){\color[rgb]{0,0,0}\makebox(0,0)[lb]{\smash{$P$}}}%
  \end{picture}%
\endgroup%

%% file: equid.pdf_tex
\begingroup%
  \makeatletter%
  \providecommand\color[2][]{%
    \errmessage{(Inkscape) Color is used for the text in Inkscape, but the package 'color.sty' is not loaded}%
    \renewcommand\color[2][]{}%
  }%
  \providecommand\transparent[1]{%
    \errmessage{(Inkscape) Transparency is used (non-zero) for the text in Inkscape, but the package 'transparent.sty' is not loaded}%
    \renewcommand\transparent[1]{}%
  }%
  \providecommand\rotatebox[2]{#2}%
  \ifx\svgwidth\undefined%
    \setlength{\unitlength}{480bp}%
    \ifx\svgscale\undefined%
      \relax%
    \else%
      \setlength{\unitlength}{\unitlength * \real{\svgscale}}%
    \fi%
  \else%
    \setlength{\unitlength}{\svgwidth}%
  \fi%
  \global\let\svgwidth\undefined%
  \global\let\svgscale\undefined%
  \makeatother%
  \begin{picture}(1,0.5)%
    \put(0,0){\includegraphics[width=\unitlength]{equid.pdf}}%
  \end{picture}%
\endgroup%

%% file: FirstIsogonicCenter.pdf_tex
\begingroup%
  \makeatletter%
  \providecommand\color[2][]{%
    \errmessage{(Inkscape) Color is used for the text in Inkscape, but the package 'color.sty' is not loaded}%
    \renewcommand\color[2][]{}%
  }%
  \providecommand\transparent[1]{%
    \errmessage{(Inkscape) Transparency is used (non-zero) for the text in Inkscape, but the package 'transparent.sty' is not loaded}%
    \renewcommand\transparent[1]{}%
  }%
  \providecommand\rotatebox[2]{#2}%
  \ifx\svgwidth\undefined%
    \setlength{\unitlength}{589.67202148bp}%
    \ifx\svgscale\undefined%
      \relax%
    \else%
      \setlength{\unitlength}{\unitlength * \real{\svgscale}}%
    \fi%
  \else%
    \setlength{\unitlength}{\svgwidth}%
  \fi%
  \global\let\svgwidth\undefined%
  \global\let\svgscale\undefined%
  \makeatother%
  \begin{picture}(1,1.0853491)%
    \put(0,0){\includegraphics[width=\unitlength]{FirstIsogonicCenter.pdf}}%
    \put(0.55366593,0.62492985){\color[rgb]{0,0,0}\makebox(0,0)[lb]{\smash{$I_1$}}}%
    \put(0.09593221,0.53173853){\color[rgb]{0,0,0}\makebox(0,0)[lb]{\smash{$A$}}}%
    \put(0.53722037,0.84420349){\color[rgb]{0,0,0}\makebox(0,0)[lb]{\smash{$B$}}}%
    \put(0.76471675,0.53173861){\color[rgb]{0,0,0}\makebox(0,0)[lb]{\smash{$C$}}}%
  \end{picture}%
\endgroup%

%% file: ABCwithArcs.pdf_tex
\begingroup%
  \makeatletter%
  \providecommand\color[2][]{%
    \errmessage{(Inkscape) Color is used for the text in Inkscape, but the package 'color.sty' is not loaded}%
    \renewcommand\color[2][]{}%
  }%
  \providecommand\transparent[1]{%
    \errmessage{(Inkscape) Transparency is used (non-zero) for the text in Inkscape, but the package 'transparent.sty' is not loaded}%
    \renewcommand\transparent[1]{}%
  }%
  \providecommand\rotatebox[2]{#2}%
  \ifx\svgwidth\undefined%
    \setlength{\unitlength}{720bp}%
    \ifx\svgscale\undefined%
      \relax%
    \else%
      \setlength{\unitlength}{\unitlength * \real{\svgscale}}%
    \fi%
  \else%
    \setlength{\unitlength}{\svgwidth}%
  \fi%
  \global\let\svgwidth\undefined%
  \global\let\svgscale\undefined%
  \makeatother%
  \begin{picture}(1,1.11111111)%
    \put(0,0){\includegraphics[width=\unitlength]{ABCwithArcs.pdf}}%
    \put(0.12475065,0.5449416){\color[rgb]{0,0,0}\makebox(0,0)[lb]{\smash{$A$}}}%
    \put(0.71346921,0.79966658){\color[rgb]{0,0,0}\makebox(0,0)[lb]{\smash{$B$}}}%
    \put(0.81949277,0.46308976){\color[rgb]{0,0,0}\makebox(0,0)[lb]{\smash{$C$}}}%
    \put(0.30304576,0.85975205){\color[rgb]{0,0,0}\makebox(0,0)[lb]{\smash{θ}}}%
    \put(0.5917432,0.53653221){\color[rgb]{0,0,0}\makebox(0,0)[lb]{\smash{$P$}}}%
    \put(0.19255466,0.80355137){\color[rgb]{0,0,0}\makebox(0,0)[lb]{\smash{$\tau_1$}}}%
    \put(0.75452541,0.65053895){\color[rgb]{0,0,0}\makebox(0,0)[lb]{\smash{$\tau_2$}}}%
    \put(0.40490141,0.18100134){\color[rgb]{0,0,0}\makebox(0,0)[lb]{\smash{$\tau_3$}}}%
  \end{picture}%
\endgroup%

%% file: ABCwithPequalsB.pdf_tex
\begingroup%
  \makeatletter%
  \providecommand\color[2][]{%
    \errmessage{(Inkscape) Color is used for the text in Inkscape, but the package 'color.sty' is not loaded}%
    \renewcommand\color[2][]{}%
  }%
  \providecommand\transparent[1]{%
    \errmessage{(Inkscape) Transparency is used (non-zero) for the text in Inkscape, but the package 'transparent.sty' is not loaded}%
    \renewcommand\transparent[1]{}%
  }%
  \providecommand\rotatebox[2]{#2}%
  \ifx\svgwidth\undefined%
    \setlength{\unitlength}{720bp}%
    \ifx\svgscale\undefined%
      \relax%
    \else%
      \setlength{\unitlength}{\unitlength * \real{\svgscale}}%
    \fi%
  \else%
    \setlength{\unitlength}{\svgwidth}%
  \fi%
  \global\let\svgwidth\undefined%
  \global\let\svgscale\undefined%
  \makeatother%
  \begin{picture}(1,1.11111111)%
    \put(0,0){\includegraphics[width=\unitlength]{ABCwithPequalsB.pdf}}%
    \put(0.18414219,0.51219205){\color[rgb]{0,0,0}\makebox(0,0)[lb]{\smash{$A$}}}%
    \put(0.58075103,0.71363292){\color[rgb]{0,0,0}\makebox(0,0)[lb]{\smash{$P = B$}}}%
    \put(0.70081909,0.44644205){\color[rgb]{0,0,0}\makebox(0,0)[lb]{\smash{$C$}}}%
    \put(0.30304576,0.85975205){\color[rgb]{0,0,0}\makebox(0,0)[lb]{\smash{θ}}}%
    \put(0.18780188,0.72945192){\color[rgb]{0,0,0}\makebox(0,0)[lb]{\smash{$\tau_1$}}}%
    \put(0.68991374,0.5970163){\color[rgb]{0,0,0}\makebox(0,0)[lb]{\smash{$\tau_2$}}}%
    \put(0.36151195,0.40079843){\color[rgb]{0,0,0}\makebox(0,0)[lb]{\smash{$\tau_3$}}}%
  \end{picture}%
\endgroup%

%% file: ABCwithPFarther.pdf_tex
\begingroup%
  \makeatletter%
  \providecommand\color[2][]{%
    \errmessage{(Inkscape) Color is used for the text in Inkscape, but the package 'color.sty' is not loaded}%
    \renewcommand\color[2][]{}%
  }%
  \providecommand\transparent[1]{%
    \errmessage{(Inkscape) Transparency is used (non-zero) for the text in Inkscape, but the package 'transparent.sty' is not loaded}%
    \renewcommand\transparent[1]{}%
  }%
  \providecommand\rotatebox[2]{#2}%
  \ifx\svgwidth\undefined%
    \setlength{\unitlength}{720bp}%
    \ifx\svgscale\undefined%
      \relax%
    \else%
      \setlength{\unitlength}{\unitlength * \real{\svgscale}}%
    \fi%
  \else%
    \setlength{\unitlength}{\svgwidth}%
  \fi%
  \global\let\svgwidth\undefined%
  \global\let\svgscale\undefined%
  \makeatother%
  \begin{picture}(1,1.11111111)%
    \put(0,0){\includegraphics[width=\unitlength]{ABCwithPFarther.pdf}}%
    \put(0.08815103,0.52479356){\color[rgb]{0,0,0}\makebox(0,0)[lb]{\smash{$A$}}}%
    \put(0.39630837,0.66650578){\color[rgb]{0,0,0}\makebox(0,0)[lb]{\smash{$B$}}}%
    \put(0.5291565,0.47284767){\color[rgb]{0,0,0}\makebox(0,0)[lb]{\smash{$C$}}}%
    \put(0.30304576,0.85975205){\color[rgb]{0,0,0}\makebox(0,0)[lb]{\smash{θ}}}%
    \put(0.47962765,0.8021772){\color[rgb]{0,0,0}\makebox(0,0)[lb]{\smash{$P$}}}%
    \put(0.06609266,0.67320225){\color[rgb]{0,0,0}\makebox(0,0)[lb]{\smash{$\tau_1$}}}%
    \put(0.75055647,0.56487167){\color[rgb]{0,0,0}\makebox(0,0)[lb]{\smash{$\tau_2$}}}%
    \put(0.21435394,0.49361089){\color[rgb]{0,0,0}\makebox(0,0)[lb]{\smash{$\tau_3$}}}%
  \end{picture}%
\endgroup%

%% file: DEFandABC.pdf_tex
\begingroup%
  \makeatletter%
  \providecommand\color[2][]{%
    \errmessage{(Inkscape) Color is used for the text in Inkscape, but the package 'color.sty' is not loaded}%
    \renewcommand\color[2][]{}%
  }%
  \providecommand\transparent[1]{%
    \errmessage{(Inkscape) Transparency is used (non-zero) for the text in Inkscape, but the package 'transparent.sty' is not loaded}%
    \renewcommand\transparent[1]{}%
  }%
  \providecommand\rotatebox[2]{#2}%
  \ifx\svgwidth\undefined%
    \setlength{\unitlength}{664.8bp}%
    \ifx\svgscale\undefined%
      \relax%
    \else%
      \setlength{\unitlength}{\unitlength * \real{\svgscale}}%
    \fi%
  \else%
    \setlength{\unitlength}{\svgwidth}%
  \fi%
  \global\let\svgwidth\undefined%
  \global\let\svgscale\undefined%
  \makeatother%
  \begin{picture}(1,1.04211793)%
    \put(0,0){\includegraphics[width=\unitlength]{DEFandABC.pdf}}%
    \put(0.0572687,0.96171637){\color[rgb]{0,0,0}\makebox(0,0)[lb]{\smash{$D$}}}%
    \put(0.91298272,0.75334109){\color[rgb]{0,0,0}\makebox(0,0)[lb]{\smash{$E$}}}%
    \put(0.46095553,0.0158322){\color[rgb]{0,0,0}\makebox(0,0)[lb]{\smash{$F$}}}%
    \put(0.07911177,0.54075255){\color[rgb]{0,0,0}\makebox(0,0)[lb]{\smash{$A$}}}%
    \put(0.68164868,0.81928263){\color[rgb]{0,0,0}\makebox(0,0)[lb]{\smash{$B$}}}%
    \put(0.80251772,0.49004695){\color[rgb]{0,0,0}\makebox(0,0)[lb]{\smash{$C$}}}%
    \put(0.3282084,0.9311394){\color[rgb]{0,0,0}\makebox(0,0)[lb]{\smash{θ}}}%
    \put(0.16632644,0.83429332){\color[rgb]{0,0,0}\makebox(0,0)[lb]{\smash{$\tau_1$}}}%
    \put(0.7616968,0.68478429){\color[rgb]{0,0,0}\makebox(0,0)[lb]{\smash{$\tau_2$}}}%
    \put(0.37795847,0.18547141){\color[rgb]{0,0,0}\makebox(0,0)[lb]{\smash{$\tau_3$}}}%
  \end{picture}%
\endgroup%

%% file: maximalSide.pdf_tex
\begingroup%
  \makeatletter%
  \providecommand\color[2][]{%
    \errmessage{(Inkscape) Color is used for the text in Inkscape, but the package 'color.sty' is not loaded}%
    \renewcommand\color[2][]{}%
  }%
  \providecommand\transparent[1]{%
    \errmessage{(Inkscape) Transparency is used (non-zero) for the text in Inkscape, but the package 'transparent.sty' is not loaded}%
    \renewcommand\transparent[1]{}%
  }%
  \providecommand\rotatebox[2]{#2}%
  \ifx\svgwidth\undefined%
    \setlength{\unitlength}{720bp}%
    \ifx\svgscale\undefined%
      \relax%
    \else%
      \setlength{\unitlength}{\unitlength * \real{\svgscale}}%
    \fi%
  \else%
    \setlength{\unitlength}{\svgwidth}%
  \fi%
  \global\let\svgwidth\undefined%
  \global\let\svgscale\undefined%
  \makeatother%
  \begin{picture}(1,1.11111111)%
    \put(0,0){\includegraphics[width=\unitlength]{maximalSide.pdf}}%
    \put(0.01481024,0.89750858){\color[rgb]{0,0,0}\makebox(0,0)[lb]{\smash{$D'$}}}%
    \put(0.91901838,0.69164007){\color[rgb]{0,0,0}\makebox(0,0)[lb]{\smash{$E'$}}}%
    \put(0.28621345,-0.00542509){\color[rgb]{0,0,0}\makebox(0,0)[lb]{\smash{$F'$}}}%
    \put(0.08883412,0.53820728){\color[rgb]{0,0,0}\makebox(0,0)[lb]{\smash{$A$}}}%
    \put(0.71346928,0.81538004){\color[rgb]{0,0,0}\makebox(0,0)[lb]{\smash{$B$}}}%
    \put(0.83520621,0.43839716){\color[rgb]{0,0,0}\makebox(0,0)[lb]{\smash{$C$}}}%
    \put(0.30304576,0.85975205){\color[rgb]{0,0,0}\makebox(0,0)[lb]{\smash{θ}}}%
    \put(0.573785,0.52970748){\color[rgb]{0,0,0}\makebox(0,0)[lb]{\smash{$P$}}}%
    \put(0.3232488,0.73179942){\color[rgb]{0,0,0}\makebox(0,0)[lb]{\smash{$\mc{O}_1$}}}%
    \put(0.74469611,0.57935903){\color[rgb]{0,0,0}\makebox(0,0)[lb]{\smash{$\mc{O}_2$}}}%
    \put(0.40406101,0.88219988){\color[rgb]{0,0,0}\makebox(0,0)[lb]{\smash{$x$}}}%
    \put(0.78118462,0.72730984){\color[rgb]{0,0,0}\makebox(0,0)[lb]{\smash{$y$}}}%
  \end{picture}%
\endgroup%

%% file: maximalTriangle.pdf_tex
\begingroup%
  \makeatletter%
  \providecommand\color[2][]{%
    \errmessage{(Inkscape) Color is used for the text in Inkscape, but the package 'color.sty' is not loaded}%
    \renewcommand\color[2][]{}%
  }%
  \providecommand\transparent[1]{%
    \errmessage{(Inkscape) Transparency is used (non-zero) for the text in Inkscape, but the package 'transparent.sty' is not loaded}%
    \renewcommand\transparent[1]{}%
  }%
  \providecommand\rotatebox[2]{#2}%
  \ifx\svgwidth\undefined%
    \setlength{\unitlength}{640bp}%
    \ifx\svgscale\undefined%
      \relax%
    \else%
      \setlength{\unitlength}{\unitlength * \real{\svgscale}}%
    \fi%
  \else%
    \setlength{\unitlength}{\svgwidth}%
  \fi%
  \global\let\svgwidth\undefined%
  \global\let\svgscale\undefined%
  \makeatother%
  \begin{picture}(1,1.25)%
    \put(0,0){\includegraphics[width=\unitlength]{maximalTriangle.pdf}}%
    \put(0.34092648,0.99222095){\color[rgb]{0,0,0}\makebox(0,0)[lb]{\smash{θ}}}%
    \put(0.26077061,1.18101181){\color[rgb]{0,0,0}\makebox(0,0)[lb]{\smash{$D$}}}%
    \put(0.89211586,0.88049137){\color[rgb]{0,0,0}\makebox(0,0)[lb]{\smash{$E$}}}%
    \put(0.10419695,0.05721722){\color[rgb]{0,0,0}\makebox(0,0)[lb]{\smash{$F$}}}%
    \put(0.5915955,1.04716659){\color[rgb]{0,0,0}\makebox(0,0)[lb]{\smash{$B$}}}%
    \put(0.65978073,0.61785183){\color[rgb]{0,0,0}\makebox(0,0)[lb]{\smash{$C$}}}%
    \put(0.1523061,0.80725532){\color[rgb]{0,0,0}\makebox(0,0)[lb]{\smash{$A$}}}%
  \end{picture}%
\endgroup%

%% file: maximalTriangle2.pdf_tex
\begingroup%
  \makeatletter%
  \providecommand\color[2][]{%
    \errmessage{(Inkscape) Color is used for the text in Inkscape, but the package 'color.sty' is not loaded}%
    \renewcommand\color[2][]{}%
  }%
  \providecommand\transparent[1]{%
    \errmessage{(Inkscape) Transparency is used (non-zero) for the text in Inkscape, but the package 'transparent.sty' is not loaded}%
    \renewcommand\transparent[1]{}%
  }%
  \providecommand\rotatebox[2]{#2}%
  \ifx\svgwidth\undefined%
    \setlength{\unitlength}{640bp}%
    \ifx\svgscale\undefined%
      \relax%
    \else%
      \setlength{\unitlength}{\unitlength * \real{\svgscale}}%
    \fi%
  \else%
    \setlength{\unitlength}{\svgwidth}%
  \fi%
  \global\let\svgwidth\undefined%
  \global\let\svgscale\undefined%
  \makeatother%
  \begin{picture}(1,1.25)%
    \put(0,0){\includegraphics[width=\unitlength]{maximalTriangle2.pdf}}%
    \put(0.34092648,0.96722106){\color[rgb]{0,0,0}\makebox(0,0)[lb]{\smash{θ}}}%
    \put(0.5915955,1.04716684){\color[rgb]{0,0,0}\makebox(0,0)[lb]{\smash{$B$}}}%
    \put(0.65978073,0.61785206){\color[rgb]{0,0,0}\makebox(0,0)[lb]{\smash{$C$}}}%
    \put(0.1523061,0.80725559){\color[rgb]{0,0,0}\makebox(0,0)[lb]{\smash{$A$}}}%
  \end{picture}%
\endgroup%

%% file: densesubmanifold.pdf_tex
\begingroup%
  \makeatletter%
  \providecommand\color[2][]{%
    \errmessage{(Inkscape) Color is used for the text in Inkscape, but the package 'color.sty' is not loaded}%
    \renewcommand\color[2][]{}%
  }%
  \providecommand\transparent[1]{%
    \errmessage{(Inkscape) Transparency is used (non-zero) for the text in Inkscape, but the package 'transparent.sty' is not loaded}%
    \renewcommand\transparent[1]{}%
  }%
  \providecommand\rotatebox[2]{#2}%
  \ifx\svgwidth\undefined%
    \setlength{\unitlength}{640bp}%
    \ifx\svgscale\undefined%
      \relax%
    \else%
      \setlength{\unitlength}{\unitlength * \real{\svgscale}}%
    \fi%
  \else%
    \setlength{\unitlength}{\svgwidth}%
  \fi%
  \global\let\svgwidth\undefined%
  \global\let\svgscale\undefined%
  \makeatother%
  \begin{picture}(1,0.75)%
    \put(0,0){\includegraphics[width=\unitlength]{densesubmanifold.pdf}}%
    \put(0.2125,0.34023948){\color[rgb]{0,0,0}\makebox(0,0)[lb]{\smash{$p_0$}}}%
    \put(0.60535713,0.16429268){\color[rgb]{0,0,0}\makebox(0,0)[lb]{\smash{$\gamma$}}}%
    \put(0.72922134,0.02392364){\color[rgb]{0,0,0}\makebox(0,0)[lb]{\smash{$\gamma_\epsilon$}}}%
    \put(-0.18799011,0.34040062){\color[rgb]{0,0,0}\makebox(0,0)[lb]{\smash{$t_0,u_0,v_0$}}}%
  \end{picture}%
\endgroup%

%% file: normalepsilon.pdf_tex
\begingroup%
  \makeatletter%
  \providecommand\color[2][]{%
    \errmessage{(Inkscape) Color is used for the text in Inkscape, but the package 'color.sty' is not loaded}%
    \renewcommand\color[2][]{}%
  }%
  \providecommand\transparent[1]{%
    \errmessage{(Inkscape) Transparency is used (non-zero) for the text in Inkscape, but the package 'transparent.sty' is not loaded}%
    \renewcommand\transparent[1]{}%
  }%
  \providecommand\rotatebox[2]{#2}%
  \ifx\svgwidth\undefined%
    \setlength{\unitlength}{640bp}%
    \ifx\svgscale\undefined%
      \relax%
    \else%
      \setlength{\unitlength}{\unitlength * \real{\svgscale}}%
    \fi%
  \else%
    \setlength{\unitlength}{\svgwidth}%
  \fi%
  \global\let\svgwidth\undefined%
  \global\let\svgscale\undefined%
  \makeatother%
  \begin{picture}(1,0.4375)%
    \put(0,0){\includegraphics[width=\unitlength]{normalepsilon.pdf}}%
    \put(0.79061913,0.16471496){\color[rgb]{0,0,0}\makebox(0,0)[lb]{\smash{$\gamma$}}}%
    \put(0.81825851,0.07050247){\color[rgb]{0,0,0}\makebox(0,0)[lb]{\smash{$\gamma_\epsilon$}}}%
    \put(0.68139717,0.26129326){\color[rgb]{0,0,0}\makebox(0,0)[lb]{\smash{$p$}}}%
  \end{picture}%
\endgroup%

%% file: typeDcases.pdf_tex
\begingroup%
  \makeatletter%
  \providecommand\color[2][]{%
    \errmessage{(Inkscape) Color is used for the text in Inkscape, but the package 'color.sty' is not loaded}%
    \renewcommand\color[2][]{}%
  }%
  \providecommand\transparent[1]{%
    \errmessage{(Inkscape) Transparency is used (non-zero) for the text in Inkscape, but the package 'transparent.sty' is not loaded}%
    \renewcommand\transparent[1]{}%
  }%
  \providecommand\rotatebox[2]{#2}%
  \ifx\svgwidth\undefined%
    \setlength{\unitlength}{640bp}%
    \ifx\svgscale\undefined%
      \relax%
    \else%
      \setlength{\unitlength}{\unitlength * \real{\svgscale}}%
    \fi%
  \else%
    \setlength{\unitlength}{\svgwidth}%
  \fi%
  \global\let\svgwidth\undefined%
  \global\let\svgscale\undefined%
  \makeatother%
  \begin{picture}(1,0.75)%
    \put(0,0){\includegraphics[width=\unitlength]{typeDcases.pdf}}%
    \put(0.2125,0.34023948){\color[rgb]{0,0,0}\makebox(0,0)[lb]{\smash{$p_0$}}}%
    \put(0.89821426,0.34040062){\color[rgb]{0,0,0}\makebox(0,0)[lb]{\smash{$t_0,u_0,v_0$}}}%
    \put(0.60535713,0.16429268){\color[rgb]{0,0,0}\makebox(0,0)[lb]{\smash{$\gamma$}}}%
    \put(0.72922134,0.02392364){\color[rgb]{0,0,0}\makebox(0,0)[lb]{\smash{$\gamma_\epsilon$}}}%
  \end{picture}%
\endgroup%

%% file: typeDcases2.pdf_tex
\begingroup%
  \makeatletter%
  \providecommand\color[2][]{%
    \errmessage{(Inkscape) Color is used for the text in Inkscape, but the package 'color.sty' is not loaded}%
    \renewcommand\color[2][]{}%
  }%
  \providecommand\transparent[1]{%
    \errmessage{(Inkscape) Transparency is used (non-zero) for the text in Inkscape, but the package 'transparent.sty' is not loaded}%
    \renewcommand\transparent[1]{}%
  }%
  \providecommand\rotatebox[2]{#2}%
  \ifx\svgwidth\undefined%
    \setlength{\unitlength}{640bp}%
    \ifx\svgscale\undefined%
      \relax%
    \else%
      \setlength{\unitlength}{\unitlength * \real{\svgscale}}%
    \fi%
  \else%
    \setlength{\unitlength}{\svgwidth}%
  \fi%
  \global\let\svgwidth\undefined%
  \global\let\svgscale\undefined%
  \makeatother%
  \begin{picture}(1,0.75)%
    \put(0,0){\includegraphics[width=\unitlength]{typeDcases2.pdf}}%
    \put(0.2125,0.34023948){\color[rgb]{0,0,0}\makebox(0,0)[lb]{\smash{$p_0$}}}%
    \put(0.89821426,0.34040062){\color[rgb]{0,0,0}\makebox(0,0)[lb]{\smash{$u_0,v_0$}}}%
    \put(0.60535713,0.16429268){\color[rgb]{0,0,0}\makebox(0,0)[lb]{\smash{$\gamma$}}}%
    \put(0.72922134,0.02392364){\color[rgb]{0,0,0}\makebox(0,0)[lb]{\smash{$\gamma_\epsilon$}}}%
    \put(0.06200989,0.34040062){\color[rgb]{0,0,0}\makebox(0,0)[lb]{\smash{$t_0$}}}%
  \end{picture}%
\endgroup%

%% file: positiveDiameter.pdf_tex
\begingroup%
  \makeatletter%
  \providecommand\color[2][]{%
    \errmessage{(Inkscape) Color is used for the text in Inkscape, but the package 'color.sty' is not loaded}%
    \renewcommand\color[2][]{}%
  }%
  \providecommand\transparent[1]{%
    \errmessage{(Inkscape) Transparency is used (non-zero) for the text in Inkscape, but the package 'transparent.sty' is not loaded}%
    \renewcommand\transparent[1]{}%
  }%
  \providecommand\rotatebox[2]{#2}%
  \ifx\svgwidth\undefined%
    \setlength{\unitlength}{429.67202148bp}%
    \ifx\svgscale\undefined%
      \relax%
    \else%
      \setlength{\unitlength}{\unitlength * \real{\svgscale}}%
    \fi%
  \else%
    \setlength{\unitlength}{\svgwidth}%
  \fi%
  \global\let\svgwidth\undefined%
  \global\let\svgscale\undefined%
  \makeatother%
  \begin{picture}(1,0.95209362)%
    \put(0,0){\includegraphics[width=\unitlength]{positiveDiameter.pdf}}%
    \put(0.05851639,0.48675001){\color[rgb]{0,0,0}\makebox(0,0)[lb]{\smash{$a$}}}%
    \put(0.2127869,0.52398771){\color[rgb]{0,0,0}\makebox(0,0)[lb]{\smash{$c(a)$}}}%
    \put(0.74209423,0.51866809){\color[rgb]{0,0,0}\makebox(0,0)[lb]{\smash{$c(b)$}}}%
    \put(0.93094263,0.48675001){\color[rgb]{0,0,0}\makebox(0,0)[lb]{\smash{$b$}}}%
  \end{picture}%
\endgroup%

%% file: negativeDiameter.pdf_tex
\begingroup%
  \makeatletter%
  \providecommand\color[2][]{%
    \errmessage{(Inkscape) Color is used for the text in Inkscape, but the package 'color.sty' is not loaded}%
    \renewcommand\color[2][]{}%
  }%
  \providecommand\transparent[1]{%
    \errmessage{(Inkscape) Transparency is used (non-zero) for the text in Inkscape, but the package 'transparent.sty' is not loaded}%
    \renewcommand\transparent[1]{}%
  }%
  \providecommand\rotatebox[2]{#2}%
  \ifx\svgwidth\undefined%
    \setlength{\unitlength}{429.67202148bp}%
    \ifx\svgscale\undefined%
      \relax%
    \else%
      \setlength{\unitlength}{\unitlength * \real{\svgscale}}%
    \fi%
  \else%
    \setlength{\unitlength}{\svgwidth}%
  \fi%
  \global\let\svgwidth\undefined%
  \global\let\svgscale\undefined%
  \makeatother%
  \begin{picture}(1,0.95209362)%
    \put(0,0){\includegraphics[width=\unitlength]{negativeDiameter.pdf}}%
    \put(0.4961789,0.76184378){\color[rgb]{0,0,0}\makebox(0,0)[lb]{\smash{$a$}}}%
    \put(0.53341666,0.37350774){\color[rgb]{0,0,0}\makebox(0,0)[lb]{\smash{$c(a)$}}}%
    \put(0.53341666,0.58363486){\color[rgb]{0,0,0}\makebox(0,0)[lb]{\smash{$c(b)$}}}%
    \put(0.49617895,0.19795847){\color[rgb]{0,0,0}\makebox(0,0)[lb]{\smash{$b$}}}%
  \end{picture}%
\endgroup%

%% file: hyperbolic1.pdf_tex
\begingroup%
  \makeatletter%
  \providecommand\color[2][]{%
    \errmessage{(Inkscape) Color is used for the text in Inkscape, but the package 'color.sty' is not loaded}%
    \renewcommand\color[2][]{}%
  }%
  \providecommand\transparent[1]{%
    \errmessage{(Inkscape) Transparency is used (non-zero) for the text in Inkscape, but the package 'transparent.sty' is not loaded}%
    \renewcommand\transparent[1]{}%
  }%
  \providecommand\rotatebox[2]{#2}%
  \ifx\svgwidth\undefined%
    \setlength{\unitlength}{2080bp}%
    \ifx\svgscale\undefined%
      \relax%
    \else%
      \setlength{\unitlength}{\unitlength * \real{\svgscale}}%
    \fi%
  \else%
    \setlength{\unitlength}{\svgwidth}%
  \fi%
  \global\let\svgwidth\undefined%
  \global\let\svgscale\undefined%
  \makeatother%
  \begin{picture}(1,0.34615385)%
    \put(0,0){\includegraphics[width=\unitlength]{hyperbolic1.pdf}}%
    \put(0.49759019,0.16578468){\color[rgb]{0,0,0}\makebox(0,0)[lb]{\smash{$o$}}}%
    \put(0.47880611,0.20171115){\color[rgb]{0,0,0}\makebox(0,0)[lb]{\smash{$\alpha$}}}%
    \put(0.84049504,0.170465){\color[rgb]{0,0,0}\makebox(0,0)[lb]{\smash{$\gamma$}}}%
    \put(0.86895658,0.19156176){\color[rgb]{0,0,0}\makebox(0,0)[lb]{\smash{$\delta$}}}%
    \put(0.42028926,0.14281098){\color[rgb]{0,0,0}\makebox(0,0)[lb]{\smash{$p$}}}%
    \put(0.91166034,0.21644898){\color[rgb]{0,0,0}\makebox(0,0)[lb]{\smash{$v$}}}%
    \put(0.90497127,0.2389643){\color[rgb]{0,0,0}\makebox(0,0)[lb]{\smash{$u$}}}%
    \put(0.78553425,0.16048277){\color[rgb]{0,0,0}\makebox(0,0)[lb]{\smash{$q$}}}%
    \put(0.56523381,0.12072036){\color[rgb]{0,0,0}\makebox(0,0)[lb]{\smash{$||\overline{op}||=r$}}}%
    \put(0.56523381,0.08266372){\color[rgb]{0,0,0}\makebox(0,0)[lb]{\smash{$||\overline{oq}||=d$}}}%
    \put(0.56523381,0.04379729){\color[rgb]{0,0,0}\makebox(0,0)[lb]{\smash{$||\overline{qt}||=||\overline{qu}||=||\overline{qv}||=R$}}}%
    \put(0.87100304,0.24703768){\color[rgb]{0,0,0}\makebox(0,0)[lb]{\smash{$\beta$}}}%
    \put(0.86920664,0.2726564){\color[rgb]{0,0,0}\makebox(0,0)[lb]{\smash{$t$}}}%
  \end{picture}%
\endgroup%

%% file: hyperbolic.pdf_tex
\begingroup%
  \makeatletter%
  \providecommand\color[2][]{%
    \errmessage{(Inkscape) Color is used for the text in Inkscape, but the package 'color.sty' is not loaded}%
    \renewcommand\color[2][]{}%
  }%
  \providecommand\transparent[1]{%
    \errmessage{(Inkscape) Transparency is used (non-zero) for the text in Inkscape, but the package 'transparent.sty' is not loaded}%
    \renewcommand\transparent[1]{}%
  }%
  \providecommand\rotatebox[2]{#2}%
  \ifx\svgwidth\undefined%
    \setlength{\unitlength}{2080bp}%
    \ifx\svgscale\undefined%
      \relax%
    \else%
      \setlength{\unitlength}{\unitlength * \real{\svgscale}}%
    \fi%
  \else%
    \setlength{\unitlength}{\svgwidth}%
  \fi%
  \global\let\svgwidth\undefined%
  \global\let\svgscale\undefined%
  \makeatother%
  \begin{picture}(1,0.34615385)%
    \put(0,0){\includegraphics[width=\unitlength]{hyperbolic.pdf}}%
    \put(0.49759019,0.16578468){\color[rgb]{0,0,0}\makebox(0,0)[lb]{\smash{$o$}}}%
    \put(0.47880611,0.20171115){\color[rgb]{0,0,0}\makebox(0,0)[lb]{\smash{$\alpha$}}}%
    \put(0.44689303,0.17043007){\color[rgb]{0,0,0}\makebox(0,0)[lb]{\smash{$\beta$}}}%
    \put(0.84049504,0.170465){\color[rgb]{0,0,0}\makebox(0,0)[lb]{\smash{$\gamma$}}}%
    \put(0.86895658,0.19156176){\color[rgb]{0,0,0}\makebox(0,0)[lb]{\smash{$\delta$}}}%
    \put(0.42028926,0.14281098){\color[rgb]{0,0,0}\makebox(0,0)[lb]{\smash{$p$}}}%
    \put(0.42434814,0.1119309){\color[rgb]{0,0,0}\makebox(0,0)[lb]{\smash{$t$}}}%
    \put(0.91166034,0.21644898){\color[rgb]{0,0,0}\makebox(0,0)[lb]{\smash{$v$}}}%
    \put(0.90497127,0.2389643){\color[rgb]{0,0,0}\makebox(0,0)[lb]{\smash{$u$}}}%
    \put(0.78553425,0.16048277){\color[rgb]{0,0,0}\makebox(0,0)[lb]{\smash{$q$}}}%
    \put(0.56523381,0.12072036){\color[rgb]{0,0,0}\makebox(0,0)[lb]{\smash{$||\overline{op}||=r$}}}%
    \put(0.56523381,0.08266372){\color[rgb]{0,0,0}\makebox(0,0)[lb]{\smash{$||\overline{ot}||=r+\epsilon$}}}%
    \put(0.56523381,0.04379729){\color[rgb]{0,0,0}\makebox(0,0)[lb]{\smash{$||\overline{oq}||=d$}}}%
    \put(0.56523381,0.00533575){\color[rgb]{0,0,0}\makebox(0,0)[lb]{\smash{$||\overline{qu}||=||\overline{qv}||=R$}}}%
  \end{picture}%
\endgroup%

%% file: poly.pdf_tex
\begingroup%
  \makeatletter%
  \providecommand\color[2][]{%
    \errmessage{(Inkscape) Color is used for the text in Inkscape, but the package 'color.sty' is not loaded}%
    \renewcommand\color[2][]{}%
  }%
  \providecommand\transparent[1]{%
    \errmessage{(Inkscape) Transparency is used (non-zero) for the text in Inkscape, but the package 'transparent.sty' is not loaded}%
    \renewcommand\transparent[1]{}%
  }%
  \providecommand\rotatebox[2]{#2}%
  \ifx\svgwidth\undefined%
    \setlength{\unitlength}{640bp}%
    \ifx\svgscale\undefined%
      \relax%
    \else%
      \setlength{\unitlength}{\unitlength * \real{\svgscale}}%
    \fi%
  \else%
    \setlength{\unitlength}{\svgwidth}%
  \fi%
  \global\let\svgwidth\undefined%
  \global\let\svgscale\undefined%
  \makeatother%
  \begin{picture}(1,1)%
    \put(0,0){\includegraphics[width=\unitlength]{poly.pdf}}%
    \put(0.64902298,0.61619305){\color[rgb]{0,0,0}\makebox(0,0)[lb]{\smash{$Q$}}}%
    \put(0.17321428,0.88035714){\color[rgb]{0,0,0}\makebox(0,0)[lb]{\smash{$\ell_1$}}}%
    \put(0.20535713,0.17678574){\color[rgb]{0,0,0}\makebox(0,0)[lb]{\smash{$\ell_2$}}}%
    \put(0.86071426,0.5357143){\color[rgb]{0,0,0}\makebox(0,0)[lb]{\smash{$\ell_3$}}}%
  \end{picture}%
\endgroup%